\documentclass[12pt]{amsart}
\usepackage{amsmath}
\usepackage{amssymb}
\usepackage{mathpazo}
\usepackage{eucal}
\usepackage{amsthm}
\usepackage{amscd}
\usepackage{hyperref}
\usepackage{fullpage}
\usepackage{url}
\usepackage{verbatim}
\usepackage{comment} 

\usepackage{amssymb,}
\usepackage[]{amsmath, amsthm, amsfonts,graphicx, amscd,}
\usepackage[all,cmtip]{xy}
\usepackage{tikz,tikz-cd} 
\usetikzlibrary{matrix,arrows}
\usepackage{xypic}
\usepackage{mathtools}
\usepackage{color}
\usepackage{enumitem}

\newcommand{\james}[1]{{\color{blue} \sf $\spadesuit\spadesuit\spadesuit$ James: [#1]}}
\newcommand{\taylor}[1]{{\color{purple} \sf $\spadesuit\spadesuit\spadesuit$ Taylor: [#1]}}

\newcommand{\todo}[1]{{\color{orange} \sf $\heartsuit\heartsuit\heartsuit$ Final Todo: [#1]}}

\begin{document}

	\newtheorem{theorem}{Theorem}[section]
	\newtheorem{lemma}[theorem]{Lemma}
 	\newtheorem{proposition}[theorem]{Proposition}
  	\newtheorem*{lemma*}{Lemma}
 	\newtheorem*{proposition*}{Proposition}
	\newtheorem{cor}[theorem]{Corollary}
	\theoremstyle{definition}
	\newtheorem{definition}[theorem]{Definition}
	\newtheorem{example}[theorem]{Example}
	\theoremstyle{remark}
	\newtheorem{remark}[theorem]{Remark}
	\numberwithin{equation}{section}
	\theoremstyle{remark}
	\theoremstyle{definition}
	\newcommand{\pd}[2]{\frac{\partial #1}{\partial #2}}
	\newcommand{\pp}{\partial }
	\newcommand{\pdtwo}[2]{\frac{\partial^2 #1}{\partial #2^2}}
	\newcommand{\od}[2]{\frac{d #1}{d #2}}
	\def\Ind{\setbox0=\hbox{$x$}\kern\wd0\hbox to 0pt{\hss$\mid$\hss} \lower.9\ht0\hbox to 0pt{\hss$\smile$\hss}\kern\wd0}
	\def\Notind{\setbox0=\hbox{$x$}\kern\wd0\hbox to 0pt{\mathchardef \nn=12854\hss$\nn$\kern1.4\wd0\hss}\hbox to 0pt{\hss$\mid$\hss}\lower.9\ht0 \hbox to 0pt{\hss$\smile$\hss}\kern\wd0}
	\def\ind{\mathop{\mathpalette\Ind{}}}
	\def\nind{\mathop{\mathpalette\Notind{}}}
	\newcommand{\m}{\mathbb }
	\newcommand{\mc}{\mathcal }
	\newcommand{\mf}{\mathfrak }
	\newcommand{\is}{^{p^ {-\infty}}}
	\newcommand{\QQ}{\mathbb Q}
	\newcommand{\fh}{\mathfrak h}
	\newcommand{\CC}{\mathbb C}
	\newcommand{\RR}{\mathbb R}
	\newcommand{\ZZZ}{\mathbb Z}
	\newcommand{\tp}{\operatorname{tp}}
	\newcommand{\SL}{\operatorname{SL}}
	\newcommand{\C}{\mathsf{C}}
	\newcommand{\Hom}{\mathrm{Hom}}
	\newcommand{\OO}{\mathcal{O}}
	\newcommand{\KS}{\mathrm{KS}}
	\newcommand{\eps}{\varepsilon}
	\newcommand{\Ext}{\mathrm{Ext}}
	\newcommand{\M}{\mathcal{M}}
	\newcommand{\N}{\mathcal{N}}
	\newcommand{\Ccal}{\mathcal{C}}
	\newcommand{\Acal}{\mathcal{A}}
	\newcommand{\Der}{\mathrm{Der}}
	\newcommand{\sect}{\S}
	\newcommand{\ZZ}{\mathbb{Z}}
	\newcommand{\Jac}{\mathrm{Jac}}
	\newcommand{\Sch}{\mathsf{Sch}}
	\newcommand{\Coh}{\mathrm{Coh}}
	\newcommand{\Mod}{\mathsf{Mod}}
	\newcommand{\Spec}{\operatorname{Spec}}
	\newcommand{\AffLin}{\mathsf{AffLin}}
	\newcommand{\AffLinu}{\underline{\mathsf{AffLin}}}
	\newcommand{\colim}{\operatorname{colim}}
	\newcommand{\id}{\operatorname{id}}
	\newcommand{\acl}{\operatorname{acl}}
	\newcommand{\trdeg}{\operatorname{trdeg}}
	\newcommand{\Khat}{\widehat{K}}
	\newcommand{\ord}{\operatorname{ord}}
	\newcommand{\Aut}{\operatorname{Aut}}
	\newcommand{\LL}{\mathbb{L}}
	\newcommand{\alg}{\operatorname{alg}}
	\newcommand{\KK}{\widehat{K}}
	\renewcommand{\AA}{\mathbb{A}}
	\newcommand{\PP}{\mathbb{P}}
	\newcommand{\ev}{\operatorname{ev}}
	\newcommand{\rk}{\operatorname{rk}}
	\newcommand{\Var}{\mathsf{Var}}
	\newcommand{\DVar}{\mbox{$D$-}\mathsf{Var}}
	\newcommand{\Kbar}{\overline{K}}
	\newcommand{\Scal}{\mathcal{S}}
	\newcommand{\DSch}{\operatorname{D-Sch}}
	\newcommand{\End}{\operatorname{End}}
	\newcommand{\dR}{\operatorname{dR}}
	\newcommand{\Hcal}{\mathcal{H}}
	\newcommand{\Wcal}{\mathcal{W}}

	\setcounter{tocdepth}{1}
	
	\title{Order one differential equations on nonisotrivial algebraic curves}
	\author{Taylor Dupuy and James Freitag}
	\email{taylor.dupuy@uvm.edu, freitagj@gmail.com}
	\maketitle

	\begin{abstract}
		In this paper we provide new examples of geometrically trivial strongly minimal differential algebraic varieties living on nonisotrivial curves over differentially closed fields of characteristic zero.
		Our technique involves developing a theory of Kodaira-Spencer forms and building connections to deformation theory. 
		In our development, we answer several open questions posed by Rosen and some natural questions about Manin kernels.
	\end{abstract}
	
	\tableofcontents

	\section{Introduction}
	
	Given a superstable theory, numerous aspects of models of the theory are controlled by the strongly minimal sets and types. There are many manifestations of this theme in various settings of model theory. Classical examples include:
	\begin{itemize} 
		\item Shelah's work characterizing prime models of totally transcendental theories \cite{shelah1972uniqueness}, 
  \item the proof of Vaught's conjecture for $\aleph_1$-categorical theories by Baldwin and Lachlan \cite{baldwin1971strongly}, 
		\item the proof of Vaught's conjecture in the $\omega$-stable setting by Shelah, Harrington and Makkai \cite{shelah1984proof}, 
		\item the fine structure results employed in characterizing uncountable spectra of countable theories by Hart, Hrushovski and Laskowski \cite{hart2000uncountable}. 
	\end{itemize}
	
Classical examples specifically from the theory of differential fields include the proof of the nonminimality of differential closure by Shelah \cite{shelah1973differentially}, Kolchin \cite{kolchin1974constrained} and Rosenlicht  \cite{rosenlicht1974nonminimality}. Later examples include Pillay's proof that a field is differentially closed if and only if it has solutions to each strongly minimal formula \cite{pillay1997differential} and the work of Hrushovski and Itai \cite{Hrushovski2003} building superstable differential fields via characterizing nonorthogonality classes of order one autonomous strongly minimal sets.

 More recently strong minimality has had a major impact on transcendence results. Nagloo and Pillay resolve numerous questions about Painlev\'e equations \cite{nagloo2017algebraic}. Freitag, Jaoui, and Moosa establish general transcendence results by reducing to the case of strongly minimal sets \cite{freitag2022any} and develop bounds on the degree of nonminimality in terms of binding groups associated with strongly minimal sets \cite{freitag2023degree, freitag2021bounding}. Structural results for strongly minimal sets have repeatedly played a decisive role in applications of model theory to algebraic differential equations. Adding to these structural results is the main motivation of this manuscript. 

Concretely, the main thrust of the paper is to solve two open problems of \cite{Rosen2007} (also mentioned in \cite{Hrushovski2003}) and generalize some work of \cite{Bertrand2010} \cite{Bertrand2016}, which gives some natural examples of Manin kernels of simple abelin varieties. This first problem is the subject of Section~\ref{tauforms} and involves finding geometrically trivial strongly minimal varieties living on nonisotrivial curves. The problem is given in \cite[see the fourth paragraph of Section 2]{Hrushovski2003} and completely specific terms in \cite[last sentence of page 2]{Rosen2007}. 

The second question of \cite{Rosen2007} we answer in Section \ref{backtothejacobian}. Rosen showed that strongly minimal sets on nonisotrivial curves could be given by the data of a certain twisted differential form on the curve, generalizing the work of Hrushovski and Itai \cite{Hrushovski2003}. The sheaf of such Kodaira-Spencer forms (which we will simply call ``KS-forms'' hereafter), has been studied in numerous works of deformation theory from various perspectives (for detailed references, see the following subsections). Rosen asked if the global sections of the sheaf of KS-forms on a curve of genus 2 or higher are in bijective correspondence with the global sections of the sheaf of KS-forms of its Jacobian:
$$H^0 (C , \Omega _C ^ \tau ) \cong H^0 ( J , \Omega _J ^ \tau ).$$ We answer this question affirmatively; our proof uses properties of Kodaira-Spencer classes, the Abel-Jacobi map, and is phrased in the language of derived categories. 

The last main contribution of the paper has to do with giving some specific examples of Manin kernels, systems of differential equations which have played an important role in diophantine applications. Manin kernels of abelian varieties of dimension $g$ are known to have order (or absolute dimension) $h$ with $g \leq h \leq 2g$. The Manin kernels of simple abelain varieties correspond precisely to the nonorthogonality classes of modular nontrivial strongly minimal sets in $DCF_0$. Despite this fact and Manin kernels being the object of intense model theoretic study over the last decades, the only example in the literature with $g < h < 2g$ for a simple abelian varieties seems to be due to Bertrand and Pillay \cite[page 504]{Bertrand2010} \cite[Proposition 2.14]{Bertrand2016} and is an Abelian variety of dimension $4$. 

Of course, understanding precisely which $h$ in the possible range might occur and in which circumstances is of great interest, since detailed versions of the Zilber trichotomy have played a central role in applications of model theory to differential equations. Our development of KS-forms and details about the moduli space of abelian varieties give examples akin to those of Bertrand and Pillay with $g < h < 2g$, but more generally in any odd dimension. 

In the next subsection, we give a more complete and technical description of each of the problems above as well as the approach and settings considered by Rosen \cite{Rosen2007} and Hrushovski and Itai \cite{Hrushovski1993}. 

\subsection{Differential equations, differential forms, and Manin kernels.}
Following the work of Hrushovski and Sokolovic \cite{hrushovski1994minimal}, there was a detailed classification of the geometries of strongly minimal sets definable in differentially closed fields with a number of notable questions regarding geometrically trivial strongly minimal sets left open. For instance, no geometrically trivial strongly minimal set with infinitely many algebraic solutions was known\footnote{See page 4268 of \cite{Hrushovski2003}} at the time. Such equations later emerged from the work of Freitag and Scanlon \cite{freitag2017strong} and Casale, Freitag, and Nagloo \cite{casale2020ax}, while Freitag and Moosa proved there are no such examples among order one equations \cite{freitag2017finiteness}. Hrushovski and Itai undertook a close study of \emph{autonomous}\footnote{A differential equation is called autonomous if it has constant coefficients.} absolute dimension one (order one)\footnote{When a system of equations, $X$, over a differential field $k$ has a generic solution such that the transcendence degree of the differential field extension generated by the solution over $k$ is $n$, we will usually say the equation has \emph{absolute dimension} $n$. In this case, $n$ is also the Kolchin polynomial of $X$ \cite{kolchin1973differential}. Some authors say $X$ has order $n$ in this situation, though this opens up possible confusion, since one sometimes talks of the order of the equations in systems with multiple independent variables. For instance, one commonly writes systems using a vector field living on some algebraic variety. In this case, the equations on the coordinate variables of the variety are each of first order, but the absolute dimension of the system is equal to the dimension of the variety.} differential equations, characterizing the nonorthogonality of two such equations living on higher genus curves in terms of morphisms of curves and pullbacks of differential forms. 

Their approach begins with a simple observation. Let $(K, \delta)$ be a differential field and let $V \subset \Khat^n$ be an algebraic variety defined over the constants of $K$; then $$\{ (\vec a , \delta \vec a) \, | \, \vec a \in V(K)\} \subseteq T_{V/K}$$ where $T_{V/K}$ denotes tangent bundle of $V$. 
	
Thus, one might think of differential equations with constant coefficients as specifying algebraic relations between functions corresponding to the coordinates of the tangent bundle of a variety; for reasons described in the introduction of \cite{Hrushovski2003}, it is, for certain problems, more natural to work with differential forms in the dual of the tangent bundle. So, Hrushovski and Itai take take up a detailed study of this approach in the case that $V=C$ is an algebraic curve, where they study pairs $(C, \omega)$ where $C$ is an algebraic curve and $\omega $ is a differential form.

The adaptation of the work of Hrushovski and Itai to non-autonomous equations is non-obvious, but was partially undertaken by Rosen \cite{Rosen2007} who developed the formalism of Kodaira-Spencer forms (KS-forms) on curves. For a variety $X/K$ with nonconstant coefficients, the sheaf of differential forms is no longer pertinent, since the derivative of a point on a variety not defined over the constants no longer lies in the tangent bundle $T_{X/K}$, but rather a torsor of the tangent bundle $\tau X$, called a \emph{prolongation} or \emph{jet space}.\footnote{A certain ``dual'' of the prolongation is what \cite{Rosen2007a} calls the sheaf of $\tau$-forms, and the approach in \cite{Rosen2007} centered on studying the properties of this sheaf in order to generalize the results of \cite{Hrushovski2003}.} The associated cohomology class of the torsor in $H^1(X,T_{X/K})$ is the Kodaira-Spencer class (reviewed in Section~\ref{S:kodaira-spencer}).

While neither the sheaf $\Omega^{\tau}_{X/K}$ nor the torsor $\tau X$ are new, the explicit use of $\Omega^{\tau}_{X/K}$ as an object ``dual'' to the torsor seems to be new (the definition of ``dual'' in the sense we are using it here can be found in Section~\ref{S:tau-forms-as-functions}). Previous uses for the sheaf of KS-forms include \cite{Buium1994,Buium1996} and \cite{Noguchi1981,Martin-Deschamps1984,Dupuy2017} where it was used to deduce cases of the Lang-Bombieri-Noguchi conjecture outside of the Mordell-Lang conjecture.\footnote{The notation varies in these sources, where the sheaf is sometimes denoted by $E_X$ rather than $\Omega^{\tau}_{X/K}$.} 

The sheaf $\Omega^{\tau}_{X/K}$ under consideration is defined for any algebraic variety $X/(K,\delta)$ and is an extension of the sheaf of differentials $\Omega_{X/L}$ by the structure sheaf $\OO_X$:
$$ 0 \to \OO_X \to \Omega^{\tau}_{X/K} \to \Omega^1_{X/K} \to 0. $$
Its class in $\Ext^1_{\OO_X}(\Omega_{X/K}^1,\OO_X)$ is, again, the Kodaira-Spencer class (here we use the isomorphism $\Ext^1_{\OO_X}(\Omega_{X/K}^1,\OO_X) \cong H^1(X,T_{X/K})$).
	
One of the central thrusts of \cite{Hrushovski2003} was to establish the existence of new geometrically trivial strongly minimal varieties living on curves of genus at least two over the constants. This is accomplished using the embedding of a curve into its Jacobian and understanding the behavior of the differentials with respect to this embedding. From a model theoretic point of view, a very natural goal is to give a characterization of the nonorthogonality classes of such strongly minimal sets. Hrushovski and Itai did this by showing such sets are in one-to-one correspondence with the rational $1$-forms on curves which are not the pullback of $1$-forms on lower genus curves - such forms are called \emph{new} or \emph{essential}. 

In particular for a curve $C/K$ with Jacobian $J$ one makes use of the natural isomorphisms $H^1(C,\OO_C) \cong H^1(J,\OO_J)$ and $H^0(C,\Omega^1_{C/K}) \cong H^0(J,\Omega^1_{J/K})$ to establish existence of \emph{essential} or \emph{new forms} --- forms on the curve that don't come from pullbacks from morphisms of other curves. 
	
A similar situation exists in the setting of \cite{Rosen2007} (see Theorem~\ref{thmrosen}), however while \cite{Rosen2007} introduces some interesting objects it does not establish the existence of any trivial strongly minimal varieties over curves not defined over the constants. Our theory allows for this, which we demonstrate with several examples of nonautonomous order one equations living on higher genus curves. 
	
Here the main obstruction to producing such examples of strongly minimal varieties seems to be due to the fact that the global geometry of KS-forms is more complicated than the global geometry of differential forms. For instance, \cite[see the second sentence of the paragraph following Proposition 3.5]{Rosen2007} asked whether for a curve $C$ of genus at least two over $K$, with Jacobian $J$, 
	\begin{equation}\label{E:global-forms}
	H^0 (C , \Omega _{C/K} ^ \tau ) \cong H^0 ( J , \Omega _{J/K} ^ \tau ).
	\end{equation}
	Here $\Omega _{X/K}^\tau $ denotes the sheaf of $\tau $ forms on $X$ for a variety $X/K$.
	The isomorphism~\eqref{E:global-forms} is Theorem~\ref{T:global-forms} of the present paper.
	To establish this result we make use of the functorial properties of the construction $X \mapsto \Omega^{\tau}_{X/K}$ and relate the global geometry of $\Omega_{C/K}^{\tau}$ to $\Omega_{J/K}^{\tau}$ using the language of derived categories.

	

Curiously, our approach involves a close association between the KS-forms on a curve and such forms on its Jacobian, and our results here allow us to resolve a number of natural questions about \emph{Manin kernels}. 
When $A$ is an Abelian variety, the smallest closed subset of the Kolchin topology which contains the torsion points of $A$ is known as the Manin kernel of $A$, denoted by $A^ \sharp$. 
These differential varieties play an essential role in numerous works in algebraic differential equations and applications, most importantly the differential algebraic proof of the Mordell-Lang conjecture in characteristic zero by Buium \cite{Buium1992a} and later the relative Mordell-Lang conjecture in characteristic $p$ by Hrushovski  \cite{hrushovski1996mordell}. Like various other types of data (e.g. Tate modules) Manin kernels determine isogeny classes; if $A$ and $B$ are simple nonisotrivial Abelian varieties then $A^{\sharp}$ and $B^{\sharp}$ have a (nonconstant) morphism between them if and only if $A$ and $B$ are isogenous \cite[Theorem 3, part 3]{Buium1997}. Of course, Manin Kernels are orthogonal to the equations coming from KS-forms on curves, so it is curious that Manin kernels play a pivotal role in our construction of these new geometrically trivial equations.

When $A$ is an abelian variety of dimension $g$, it is known that the absolute dimension of $A^ \sharp $ is between $g $ and $2g$, with the lower bound achieved precisely when $A$ is isotrivial. The absolute dimension of the Manin kernel $A^\sharp$ is $g$ plus the Kodaira-Spencer rank of $A$.
Buium shows that for each $g$ and $g \leq k \leq 2g$, there is an Abelian variety of dimension $g$ whose Manin kernel has absolute dimension $k$ \cite[pg 214]{Buium1997}  \cite[Theorem 6.1]{Buium1993}, while the generic points in the moduli space correspond to Abelian varieties whose Manin kernels have absolute dimension $2g$  \cite[Proposition 6.7]{Buium1993}.
	
We say Abelian varieties whose Manin kernel has absolute dimension less than $2g$ and larger than $g$ have \emph{intermediate rank}. To produce Abelian varieties whose Manin kernels have intermediate rank, Buium uses products of elliptic curves - for elliptic curves the problem is much simpler; isotrivial elliptic curves have Manin kernels with absolute dimension one and nonisotrivial curves have Manin kernels of dimension two. By using details about the moduli space of abelian varieties, we show that for each odd $g$, there is a simple abelian variety whose Manin kernel has absolute dimension $k$ with $g < k < 2g$. 
	
Manin kernels of nonisotrivial simple abelian varieties correspond precisely to the nonorthogonality classes of modular nontrivial strongly minimal types. Detailed versions of the Zilber trichotomy for strongly minimal sets have played a major role in applications of the model theory of differential fields. For instance, once an equation of order larger than one which is defined over the constants is seen to be strongly minimal, one can use the fact that there are no modular nontrivial strongly minimal sets defined over the constants to deduce its geometric triviality. See, for instance \cite{freitag2017strong, casale2020ax}. So, there is strong motivation to characterize which absolute dimensions can appear for such Manin kernels; any restriction (e.g. even over specific subfields) would likely have interesting applications. For a simple example, suppose that there was no nonisotrivial Abelian surface with Manin kernel of absolute dimension $3$ (the options are $3$ or $4$). Then any strongly minimal set of absolute dimension $3$ would automatically be geometrically trivial; this would likely be a powerful tool for certain classes of equations. However, we conjecture to the contrary that any $n \geq 2$ can appear as the absolute dimension of the Manin Kernel of a simple abelian variety.\footnote{In forthcoming work joint with Scanlon which began after the initial appearance of this manuscript, the the authors prove this conjecture over a differentially closed field with a nonconstructive proof. More constructive proofs like those in this paper over specific finitely generated fields are work in progress.}

The only previous result we know giving a simple Abelian variety whose Manin kernel has intermediate rank seems to be (somewhat indirectly) given by Bertrand and Pillay \cite[page 504]{Bertrand2010} \cite[Proposition 2.14]{Bertrand2016} and is an Abelian variety of dimension $4$. Our constructions give examples in any odd dimension. 

We use results of Buium to show that for smooth projective curves $C/K$ have the identity
$$h^0(C,\Omega^{\tau}_{C/K}) = 2g+1 - a(\Jac_C^{\sharp}),$$ 
here $a(\Jac_C^{\sharp})$ is the absolute dimension of the Manin kernel of the Jacobian of the curve. After this is established we discuss how existence of global section of $\Omega^{\tau}_{C/K}$ relate to descent. We show, building on work on Andr\'e \cite{Andre2017} there indeed exist simple abelian varieties with $0<\rk(\KS_{A/K}(\delta)) < g$, giving examples in every odd dimension of simple abelian varieties with Manin kernels of intermediate rank. 

Building on results of Noot-de Jong \cite{Jong1991}, Moonen \cite{Moonen2010}, and Frediani-Penegini-Porru \cite{Frediani2015} we show that curves with simple Jacobians with intermediate Kodaira-Spencer rank exist. Using our earlier developments, we deduce the existence of strongly minimal sets living on higher genus nonisotrivial curves, answering the question of Rosen. Our examples rely on families of curves with simple CM Jacobians. Interestingly, the Coleman Conjecture says that there should not exist any families of simple CM Jacobians for $g \geq 8$. This would then imply that the differential equations we construct in this way can only exist for $g \leq 7$ \cite{Moonen2022}. 

\subsection{Analog in the arithmetic setting.} Part of our motivation to develop the theory of KS-forms is that the analog of the sheaf of KS-forms in arithmetic settings has also seen numerous applications. The analogous forms are called \emph{Frobenius-Witt differentials}. These too have a well developed connection to deformation theory via \emph{Deligne-Illusie classes} (see e.g. \cite{Dupuy2019a} and \cite{Buium1995}). 

These were first used by Buium \cite[pg 356]{Buium1996a} in his proof of the Manin-Mumford conjecture in the arithmetic setting. The first author with Katz, Rabinoff, and Zureick Brown in \cite{Dupuy2019} further developed this theory in the context of the duality investigated in this paper and this sheaf is further studied \cite{Saito2022}.

	\subsection*{Acknowledgements}
	The authors began working on this project during the Spring of 2014 at MSRI (DMS-1440140). The current version replaces a 2017 preprint \url{https://arxiv.org/abs/1707.08714}. We thank Aaron Royer for supporting our efforts to update this manuscript and many helpful conversations around derived categories. While working on this project Dupuy was supported by
	the European Reseach Council under the European Unions Seventh Framework
	Programme (FP7/2007-2013)/ ERC Grant agreement no. 291111/MODAG. Freitag was supported by the National Science Foundation, grants DMS-1204510, DMS-1700095 and NSF CAREER award 1945251. Much of the approach of this paper is informed by the work of Rosen \cite{Rosen2007,Rosen2007a}, and numerous of his results are proved and cited below.

\section{Generalities and background}

	\subsection{The first jet space}\label{The first jet space}\label{S:first-jet-space}
	In this paper, by an \emph{algebraic variety} (or just variety) is an algebraic set over $K$, which is always assumed to be a field of characteristic zero.\footnote{Usually $K$ will also be a differential field, as we mention below.} When $V$ is a variety and $U \subset V$ is open, we often view elements $f\in \OO(U)$ as maps $f:U \to K$. 
 
	Let $(K,\delta)$ be a differential ring. 
	For a scheme $X/K$ the first jet space $\tau X$ is a ``twisted version'' of the tangent bundle $T_{X/K}$. 
	It is a torsor under $T_{X/K}$ whose local sections parametrize derivations on the structure sheaf prolonging the derivation $\delta$ on $K$. We will now review its construction. 
	
	We first give the definition for affine schemes. 
	If $A = K[x_1,\ldots,x_n]/(f_1,\ldots,f_e)$, and $X = \Spec A$ we have
	\begin{equation}\label{E:jet-ring}
	A^1 = K[x_1,\ldots, x_n, \dot x_1, \ldots, \dot x_n]/(f_1,\ldots,f_e, \delta(f_1),\ldots,\delta(f_e)), \quad \tau X = \Spec A^1.
	\end{equation}
	where $\delta(f_i)$ are expanded according to the rules
	$$ \delta(a+b) = \delta(a)+\delta(b), \quad \delta(ab) = \delta(a)b + a \delta(b), \quad \delta(c) = \delta_0(c),$$
	where these rules are understood to hold for $a,b\in K[x_1,\ldots, x_n, \dot x_1, \ldots, \dot x_n]/(f_1,\ldots,f_e)$ and $c \in K$.
	Here we used the notation where $\dot x_i = \delta(x_i)$.
	Note that there is a universal derivation 
	$$\delta_{univ}: A \to A^1,$$
	which send an element $h\in A$ to its formal derivative.
	
	For general schemes, one can check that this definition in the affine situation localizes well allowing us to define the sheaf of rings $\OO_X^{[1]}$ on $X$ analogous to how $A^1$ was defined. 
	One then defines 
	$$\tau X = \Spec_X \OO_X^{[1]},$$
	where $\Spec_X$ is the global $\Spec$ functor. 
	For details on this construction we refer the reader to \cite[Chapter 3]{Buium1994a}.
	Finally, we remark that when the derivation $\delta$ on $K$ is trivial, that this construction defines the underlying scheme of the tangent bundle.
	
	The first jet space has the convenient property that order one differential equations in the coordinate functions of  $X$ correspond to algebraic subvarieties in $\tau X$. If $\Sigma \subset X$ is an order one differential variety then there is a corresponding algebraic variety $\Sigma^{[1]} \subset \tau X$ (cut out by polynomials in the indeterminates and the indeterminates primed) and for any $K$-algebra $L$
	$$ \Sigma(L) = \Sigma^{[1]}(L) \cap \exp_1(X(L)).$$
	For each such $L$ the map $\exp_1:X(L) \to \tau X(L)$ is a non algebraic map which takes a points in the base, takes the derivative of their coordinates and obtains points in the prolongation space. One can check that this construction is actually independent of coordinates. For details we refer the reader to \cite[Section 3.8]{Buium1994} (there Buium uses the notation $\nabla_1$ in place of $\exp_1$). 
	
	We now review the torsor structure on the first jet space. 
	Let $X$ be a variety over a differential ring $(K,\delta)$, and suppose $\delta_1$ and $\delta_2$ are derivations on the structure sheaf of $X$ extending the derivation on $K$. 
	One can see that their difference is $K$-linear so that $\delta_1 - \delta_2 \in T_{X/K}(U)$. 
	This makes $\tau X$ a torsor under $T_{X/K}$ in the category of schemes over $X$. 
	The classifying element of $H^1(X,T_{X/K})$ for this torsor is called the Kodaira-Spencer class which we describe in the next section. See also  \cite[page 65]{Buium1994a}. 
	
	\subsection{Review of Kodaira-Spencer classes}\label{S:kodaira-spencer}
	Let $\pi:Z \to S$ be a morphism of schemes and let $\Omega_Z$ denote the cotangent sheaf.
	We remind the read that the $\OO_Z$-dual is the sheaf $T_Z$.
	Recall (\cite[Tag 01UM - Lemma 28.31.9]{StacksProject}) that for such a morphism $\pi$N we have the relative cotangent sequence 
	\begin{equation}\label{E:cotangent-sequence}
	\pi^*\Omega_{S} \to \Omega_{Z} \to \Omega_{Z/S} \to 0,
	\end{equation}
	which defines the relative cotangent bundle $\Omega_{Z/S}$ as the cokernel of the map 
	\begin{equation}\label{E:map-of-cotangents}
	c_{\pi}: \pi^*\Omega_S \to \Omega_Z.
	\end{equation}
	When $\pi$ is smooth the cotangent sequence \eqref{E:cotangent-sequence} extends to the left by $0$. 
	For future reference we recall that the map \eqref{E:map-of-cotangents} is defined by 
	\begin{equation}\label{E:map-of-cotangents-explicit}
	c_{\pi}: \pi^*d_S(h) \mapsto d_S(\pi^{*}h) 
	\end{equation}
	for $h$ a local section of $\OO_S$ (\cite[Tag 01UNM - Lemma 28.31.8]{StacksProject}). 
	A morphism for relative differentials is similarly defined and we denote it by $c_{\pi/S}$.
	
	\begin{definition}
		The \emph{(full) Kodaira-Spencer class}
		\begin{equation}\label{E:full-kodaira-spencer}
		\kappa_{\pi} \in \Ext^1_{\OO_Z}(\Omega_{Z/S}^1, \pi^*\Omega_S) \cong H^1(Z,T_{Z/S}\otimes_{\OO_Z} \pi^*\Omega_S)
		\end{equation}
		to be the class associated to the extension \eqref{E:cotangent-sequence} when $\pi$ is smooth. 
	\end{definition}
	Sometimes we use the notation $\kappa_{Z/S}$ for $\kappa_{\pi}$.
	\begin{definition}
		The \emph{Kodaira-Spencer map} $\KS_{\pi}: T_S \to R^1\pi_*T_{Z/S} $
		is the image of $\kappa_{Z/S}$ under the natural map 
		\begin{equation}\label{E:kodaira-spencer-map-map}
		H^1(Z,T_{Z/S}\otimes_{\OO_Z} \pi^*\Omega^1_{S}) \to \Hom_{\OO_S}(T_S, R^1\pi_*T_{Z/S}).
		\end{equation}
	\end{definition}
	This map \eqref{E:kodaira-spencer-map-map} uses two things: the projection formula (\cite[Tag 01E6]{StacksProject})
	$R^1\pi_*(T_{Z/S} \otimes_{\OO_X} \pi^*\Omega^1_S) \cong R^1\pi_*T_{Z/S} \otimes_{\OO_S} \Omega^1_S,$
	and the third map appearing in the five term exact sequence 
	$$ 0 \to H^1(S,\pi_*\mathcal{F}) \to H^1(X,\mathcal{F}) \to H^0(S,R^1\pi_*\mathcal{F}) \to H^2(S,\pi_*\mathcal{F}) \to H^2(X,\mathcal{F}) $$
	associated to the Grothendieck spectral sequence for the composition $\Gamma_X = \Gamma_S\circ \pi_*$. Here $\Gamma_{(-)}$ denotes taking global sections and $\mathcal{F} = T_{Z/S} \otimes \pi^*\Omega^1_S$.
	
	The Kodaira-Spencer map has the following description, which follows from a \u{C}ech description of the connecting homomorphism in a long exact sequence. 
	If $U \subset S$ is affine open and $\delta \in T_S(U)$ is a derivation then $\KS_{\pi}(\delta) \in H^1(\pi^{-1}(U),T_{\pi^{-1}(U)/U})$ is given explictly as follows:
	For $(V_i\to U)_{i\in I}$ a Zariski cover of $\pi^{-1}(U)$ by affine opens and $\delta_i \in T_X(V_i)$ lifts of $\delta \in T_S(U)$ (which exist by considering the relative tangent sequence on affine opens or using the infinitesimal lifting property) then $ \delta_i - \delta_j \in T_{X/S}(V_i\cap V_j) $ and hence define the cocycle $(\delta_i-\delta_j)_{(i,j)\in I\times I}$ and hence 
	\begin{equation}\label{E:cech-kodaira-spencer}
	\KS_{\pi}(\delta):=[\delta_i-\delta_j] \in H^1(\pi^{-1}(U),T_{\pi^{-1}(U)/U}).
	\end{equation}
	One can check that this construction is independent of the choices made.
	
	The special case we consider is when $S = \Spec(K)$ for $(K,\delta)$ our base differential field.
	Here we will have $Z/K$ smooth and one may lift the derivation $\delta$ to derivations $\delta_i$ on $\OO(V_i)$ where $(V_i \to Z)_{i\in I}$ is a Zariski open cover.
	The cohomology class $\KS_{Z/K}(\delta):=[\delta_i - \delta_j] \in  H^1(Z, T_{Z/K}),$ exactly as in \eqref{E:cech-kodaira-spencer}.
	This defines a map $\KS_{Z/K}: \Der(K) \to H^1(Z, T_{Z/K})$.
	
	We remark that if $S$ is an integral scheme $K = \kappa(S) = \kappa(\eta)$ is the function field of $S$ (here $\eta$ denotes the generic point) then 
	$$(\KS_{\pi})_{\eta} \mapsto \KS_{Z_K/K}$$
	Note that for every open set $U\subset S$ we have $\eta \in U$ and a commutative diagram
	$$ \xymatrix{
		T_S(U) \ar[rr]^-{(\KS_{\pi})_V} \ar[d] & & H^1(\pi^{-1}(U),T_{Z/S}) \ar[d] \\
		\Der(K) \ar[rr]^-{\KS_{Z_K/K}} & & H^1(Z_K,T_{Z_K})
	}
	$$
	where the vertical arrows are restriction maps. 
	The restriction maps commute on the various open sets commute and we get a commutative diagram:
	$$ \xymatrix{
		T_{S,\eta} \ar[rr]^-{(\KS_{\pi})_{\eta}} \ar[d] & & (R^1\pi_*T_{Z/S})_{\eta} \ar[d] \\
		\Der(K) \ar[rr]^-{\KS_{Z_K/K}} & & H^1(Z_K,T_{Z_K})
	}.
	$$
	(Here one just recalls that for a sheaf $\mathcal{G}$ on $Z$ we have $\mathcal{G}_{z} = \colim_U G(U)$ where the colimit (direct limit) varies over open sets $U\owns z$).
	
	\begin{theorem}[Kodaira-Spencer Compatibility]\label{T:functorality-of-KS}
		Let $\pi_X:X \to S$ and $\pi_Y:Y\to S$ be morphisms of schemes.
		Let $f: X \to Y$ be a smooth morphism of $S$-schemes. 
		We have 
		$$ f^* \kappa_{Y/S} = df \kappa_{X/S} \in H^1(X,f^*T_{Y/S} \otimes_{\OO_X} \pi_X^*\Omega_S).$$
	\end{theorem}
	This formula is stated (but not proved) for example in \cite[Introduction]{Faltings1999}.
	
	The proof makes use of derived categories which we give below and was worked out in conjunction with Aaron Royer while we were at MSRI.
	We refer the reader to \cite{Caldararu2005} for a quick exposition of derived categories. 
	A treatment of triangulated categories can be found in the Stacks Project as well \cite[Tag 05QK - Definition 13.3.2 Defines Triangulated Categories]{StacksProject} and Weibel \cite{Weibel1995}.
	For our purposes, an important fact is that given objects $F$ and $G$ in the bounded derived category  $D^b(X)$ of coherent $\OO_X$-modules. 
	The relationship between derived categories and triangulated categories is that Derived categories are triangulated \cite[Corollary 10.4.3]{Weibel1995}.
	(\cite[Lemma 28]{DTC2})\footnote{\cite[Tag 06XQ]{StacksProject} defines Ext functors in terms of derived categories. }
	\begin{equation}\label{derived and ext}
	\Hom_{D^b(X)}(F,G[1]) \cong \Ext_X^1(F,G).
	\end{equation}
	In the special case where we are given an extension of coherent sheaves
	$$ 
	\xymatrix{
		0 \ar[r] & G \ar[r] & E \ar[r] & F \ar[r] & 0.\\
	}
	$$
	there is an associated distinguished triangle in the derived category
	$$ 
	\xymatrix{
		& G \ar[r] & E \ar[r] & F. \\
	}
	$$
	This triangle may be ``rotated'' to a new triangle (\cite[11.1.2; T2]{Weibel1995})
	$$ 
	\xymatrix{
		& E \ar[r] & F \ar[r] & G[1] \\
	}
	$$
	and the morphism $F \to G[1]$ corresponds to the class of the extension under the isomorphism in (\ref{derived and ext}). What's more, $E$ is completely and functorially determined (in the derived category) by the map $F \to G[1]$ (\cite[Remark 10.2.2]{Weibel1995}).

	\begin{proof}[Proof of Kodaira-Spencer Compatibility]
		Pulling back the relative cotangent sequence on $Y$ to $X$ gives an exact sequence of vector bundles
		\begin{equation}\label{E:pulled-back-sequence}
		0 \to \pi_X^* \Omega_S \to f^*\Omega_Y \to f^*\Omega_{Y/S} \to 0.
		\end{equation}
		Since the pullback of an extension class is the extension class of the pullback, the class
		$$ f^*\kappa_{X/S} \in \Ext_X^1(f^*\Omega_{Y/S},\pi_X^*\Omega_S)$$
		is the class of \eqref{E:pulled-back-sequence}. 
		We have 
		$$\xymatrix{
			0 \ar[r] & \pi_X^* \Omega_S \ar[d] \ar[r] & f^*\Omega_{Y} \ar[r] \ar[d]^{c_f}& f^*\Omega_{Y/S} \ar[d]^{c_{f/S}} \ar[r] & 0\\
			0 \ar[r] & \pi_X^*\Omega_S \ar[r] & \Omega_X \ar[r] & \Omega_{X/S} \ar[r] & 0 \\
		}.$$
		The maps $c_f$ and $c_{f/S}$ are given in equation \eqref{E:map-of-cotangents}.
		Using 
		$$ \kappa_{X/S} \in \Ext_X^1(\Omega_{X/S},\pi_X^*\Omega_A) \cong \Hom_{D^b(X)}(\Omega_{X/S},\pi_X^*\Omega_S[1]) $$
		$$f^*\kappa_{Y/S} \in \Ext_X^1(f^*\Omega_{Y/S},\pi_X^*\Omega_A) \cong \Hom_{D^b(X)}(f^*\Omega_{Y/S},\pi_X^*\Omega_S[1]) $$
		we have the following commutative square in $D^b(X)$.
		$$\xymatrix{
			f^*\Omega_{Y/S} \ar[d]_{c_{f/S}}\ar[rr]^-{f^*\kappa_{Y/S} } && \pi_X^*\Omega_S[1] \ar[d]^{\id}  \\
			\Omega_{X/S} \ar[rr]^-{ \kappa_{X/S}} && \pi_X^*\Omega_S[1]
		}$$
		This implies 
		$$ df(\kappa_{X/S}) := \kappa_{X/S} \circ c_{f/S} = f^*\kappa_{Y/S}. $$
	\end{proof}
	
	In addition to this compatibility the Kodaira-Spencer class $H^1(X,T_{X/K})$ is the classifying class for $\tau X$ as a torsor under $T_{X/K}$.

 \subsection{Model theory and differential equations}
	
	Given an order one differential equation $X$ and a finitely generated differential field $K$ containing the field of definition of $X$ and a collection of solutions $f_1, \ldots , f_n$ of $X$, we often consider the algebraic closure of $f_1, \ldots , f_n$ in $X$ over $K$. By this, we mean solutions of $X$ which are algebraically dependent over $K (f_1, \ldots , f_n)$. There is a strong dichotomy for order one differential equations  - either: \begin{enumerate} 
		\item $X$ can be put into finite-to-finite correspondence with the field of constants by a differential algebraic correspondence defined over some differential field $K_1$ \emph{or} 
		\item the algebraic closure of $f_1, \ldots , f_n$ in $X$ over $K$ is the union of the algebraic closures of the individual solutions $f_1, \ldots , f_n$. 
	\end{enumerate} 
	
	Equations satisfying the second property are called \emph{geometrically trivial}; this dichotomy follows from the strong version of the Zilber trichotomy proved by \cite{hrushovski1994minimal}. In \cite[Theorem 6.2]{freitag2017finiteness}, the second condition was strengthened for order one equations - it was further established that the algebraic closure of $f_i$ in $X$ over $K$ is finite. This is equivalent to the model theoretic condition of being $\aleph _0$-categorical\footnote{This is equivalent to the theory of the reduct of $DCF_0$ to the strongly minimal set $X$ being $\aleph _0$-categorical.}. In this case, $X$ has a definable quotient $Y$ which is an order one differential equation with the property that the algebraic closure of a solution $g$ in $Y$ over $K$ consists of only $g$ together with elements from $K^{alg}$. The existence of this quotient $Y$ follows from elimination of imaginaries in differentially closed fields - see the next lemma for a formal statement in terms of differential varieties. To put this another way, all solutions of $Y$ over any differential field $K$ which are transcendental over $K$ are algebraically independent over $K$. In this case, we call the equation $Y$ \emph{disintegrated}. So, every geometrically trivial order one differential equation is \emph{nonorthogonal} to one which is disintegrated.\footnote{\emph{Nonorthogonality} is a general model theoretic notion, which, for strongly minimal differential varieties (our setting) can be equivalently stated in a natural manner: $X$ and $Y$ are nonorthogonal if there is a differential algebraic subvariety $Z \subset X \times Y$ which projects dominantly onto both $X$ and $Y$ with the property that the generic fiber of either projection is finite.} It is further true, by results of \cite{freitag2017finiteness} that order one disintegrated differential equations over a finitely generated differential field $K$ have finitely many solutions in $K$; removing these points, all solutions generate isomorphic differential field extensions of $K$. In model theoretic terms, removing these points, the equation isolates a type.  
 
 We will also sometimes use the fact that the theory of differentially closed field of characteristic zero \emph{eliminates imaginaries}, see \cite{Marker1995}, which we state next in differential algebraic terms. When $X$ is a differential algebraic variety over $K$ of finite absolute dimension and $a$ is a generic point on $X$ over $K$, then $K \langle a \rangle $ is a finitely generated field extension of $K$ - the function field of an algebraic variety $V$. In this case, following \cite[page 4270]{Hrushovski2003}, we say $X$ \emph{lives on} $V$. In this case $X$ is (perhaps discarding a proper Kolchin closed subset) equal to $\{v \in V \, | \, \delta (v) = s(v) \}$ where $s$ is a section of $\tau V$. 
 
 In \cite{Hrushovski2003} this functor from differential varieties to algebraic varieties with a section of the prolongation is explained in detail (pages 4270-4273). 

	\begin{lemma}\label{L:EIsoup} Let $X$ be an irreducible Kolchin closed set living on $V$. 
		Let $E$ be a definable equivalence relation on $X$ with generically finite classes. 
		Then there is a variety $W$ and a rational map 
		$$f:V \rightarrow W$$ 
		and a Kolchin closed set $Y$ living on $W$, a Kolchin closed set $V_{0} \subset V$, such that $f(X\setminus V_0) \subset Y$ and $f\vert_{X}:X \to f(X)$ is the quotient map.
		
		Furthermore, if $L_0$ is a differential field over which the $X,E,V$ are defined, then $V_{0},Y,W,f$ may be defined over $L_0$ as well. 
	\end{lemma}

	\section{Kodaira-Spencer forms and pairings}
	
	
	This section builds the theory of Kodaira-Spencer forms on curves, following an approach similar to that of Rosen \cite[sec 2]{Rosen2007}.

	\subsection{Kodaira-Spencer forms}\label{S:tau-forms-as-functions}
	
	In what follows, it will be useful to recall that for a scheme $X$ over a ring $K$ that the tangent bundle $T_{X/K}$, as a scheme, is isomorphic to $\Spec_X( S(\Omega_{X/K})$ where $\Spec_X$ denotes global $\Spec$ and $S(\Omega_{X/K}) = \bigoplus_{n\geq 0} \Omega_{X/K}^{\otimes n}$ is the symmetric algebra of the modules of differentials.
	If we let $\pi$ denote the structure morphism $\pi: T_{X/R} \to X$ then this means that for an open set $U\subset X$ we have $\Omega_{X/K}(U) \subset \OO_{T_{X/K}}(U)$. 
	
	From this perspective, we can see two things that we would like to highlight. 
	The first is that $\Omega_{X/K}$ is a subsheaf of $\pi_*\OO_{T_{X/K}}$ with $dx$'s playing the role of $x'$-coordinates usually found in equations for the tangent bundle. 
	The second is that for an open set $U\subset X$ there is a pairing 
	$$T_{X/K}(U) \times \Omega_{X/K}(U) \to \OO_X(U)$$ 
	which is just evaluation of functions at a point.
	In the same way that we let $dx = x'$ for the tangent bundle and $x$ a local section of $\OO_X$ we will let $d^{\tau}x = x'$ for a local section $x$ of $\OO_X$ in order to define a sheaf $\Omega_{X/K}^{\tau}$.
	
	\begin{definition}\label{tau forms}
		Let $X$ be a scheme over a differential ring $(K,\delta)$.
		The sheaf of \emph{Kodaira-Spencer forms} $\Omega^{\tau}_{X/K}$ is defined to be the sheaf associated to the presheaf 
		$$ U\mapsto \Omega^{\tau}_{X/K}(U)$$
		where $\Omega^{\tau}_{X/K}(U)$ is the $\OO_X(U)$-module spanned by the set of functions $\lbrace d^{\tau}f: f \in \OO_X(U) \rbrace \subset \OO(\tau U)$. 
	\end{definition}
	
	Let $A$ be a ring over a differential ring $(K,\delta)$. 
	The module $ \Omega_{A/K}^{\tau} $ is the free $A$-module on the symbols $\lbrace 1 \rbrace \cup \lbrace d^{\tau} a : a \in A\rbrace$ modulo the relations 
	\begin{eqnarray*}
		d^\tau(ab) &=& d^\tau(a) b + a d^\tau(b), \\
		d^\tau(a+b) &=& d^\tau(a)+d^\tau(b), \\
		d^\tau(c) &=& \delta(c),
	\end{eqnarray*}
	for $a,b\in A$ and $c \in K$.
	This description localizes well which is what allows us to sheafify.
	
	In terms of the structure sheaf of the first just space, we have $d^\tau x = x'$ for each local section $x \in \OO_X$ which gives the explicit description
	$$\Omega^{\tau}_{X/K} = F^1\OO_X^{[1]},$$
	where $F^1\OO_X^{[1]}$ is the degree less than or equal to one piece of the first jet space (see \cite[Introduction]{Buium1994}).
	Here $\OO^{[1]}_X = \pi_*\OO_{\tau X}$ where $\pi:\tau X \to X$ is the natural map and for $d\geq 0$ we define $F^d \OO_{X}^{[1]}$ as the $\OO_X$-module generated elements of the form $\delta(f_1) \delta(f_2) \cdots \delta(f_n)$ where $n\leq d$.
	In what follows we will let $\iota: \Omega^{\tau}_{X/K} \to \OO_{X}^{[1]}$ denote the canonical inclusion just described.

	\subsection{Extension description of $\Omega^{\tau}_{X/K}$ }
	We now present a description of  $\Omega^{\tau}_{X/K}$ as an extension of the sheaf of differential forms. 
	This construction follows Buium (see e.g. \cite[page 799]{Buium1994}).
	Let $X$ be a variety over a differential ring $(K,\delta)$. 
	Let $\KS_X(\delta) \in H^1(X, T_{X/K})$ be the Kodaira-Spencer class associated to the derivative on the base. 
	We will write $E_X$ which is the abstract extension 
	\begin{equation}\label{E:exact-seq}
	\xymatrix{
		0 \ar[r] & \OO_X \ar[r]^{\iota} & E_X \ar[r]^{\lambda} & \Omega_{X/K}^1 \ar[r] & 0  
	},
	\end{equation}
	associated to the Kodaira-Spencer class $\KS(\delta) \in H^1(X, T_{X/K}) \cong \Ext^1(\Omega_{X/K}, \OO_X)$. 
	The functorality properties of the Kodaira-Spencer class allow us to define pullbacks 
	\begin{equation}\label{E:pullback-morphism}
	\phi^* E_W \to E_V
	\end{equation}
	for morphisms $\phi: V \to W$ of $K$-schemes as Ext classes of pullbacks/pushouts of extensions of sheaves are pullbacks/pushouts of Ext classes of extensions of sheaves.
	This extension description allows us to circumvent many of Rosen's computations used to claim the morphisms in \eqref{E:pullback-morphism} exist and are functorial.
	
	\begin{theorem}\label{T:extension-description}
		For any variety $X$ over a differential ring $(K,\delta)$ 
		We have a functorial isomorphism $E_X \cong \Omega^{\tau}_{X/K}$. 
	\end{theorem}
	\begin{proof}
		In what follows we use that $F^1\OO^{[1]}_{X} = \Omega^{\tau}_{X/K}$. 
		The description of $\Omega^{\tau}_{X/K}$ is given in \cite[Proposition 1.3 (1)]{Buium1994} in the special case that $M = \OO_X, d=1, r=1$. 
		In loc. cit. they show that $F^1\OO^{[1]}_{X}$ fits into an exact sequence
		$$ 0 \to \OO_X \to F^1\OO^{[1]}_X \to \Omega_{X/K} \to 0, $$
		and that the boundary map of the exact sequence $H^0(X,\Omega_{X/K}) \to H^1(X,\OO_X)$ is cup product with Kodaira-Spencer $\KS(\delta) \in H^1(X,T_{X/K})$.
		Using $H^1(X, T_{X/K}) = \Ext^1(\Omega_{X/K},\OO_X) = \Hom_{D^b(X)}(\Omega_{X/K},\OO_{X}[1])$ and that cup product in sheaf cohomology coincides with 
		morphisms in the derived category, we are done. 
	\end{proof}
	
	Using the splitting we explain now how trivialization of the usual module of differentials relate to trivializations of KS-differentials.
	
	\begin{lemma}\label{L:basis} 
		Let $X$ be an integral variety over a differential ring $(K,\delta)$ of relative dimension $n$.
		\begin{enumerate}
			\item Let $U \subseteq X$ be an open subset of a variety where the sequence \eqref{E:exact-seq} splits and for which $\Omega_{X/K}(U)$ is free.
			If $\eta_1,\ldots,\eta_m\in \Omega^{\tau}_{X/K}(U)$ reduces to a basis for $\Omega_{X/K}^1(U)$ as an $\OO(U)$-module then 
			$$\Omega_{X/K}^{\tau}(U) \cong \OO(U) + \OO(U)\eta_1 + \cdots + \OO(U)\eta_m $$
			where the sum is free. 
			\item If $\eta_1,\ldots,\eta_n \in \Omega^{\tau}_{K(X)/K}$ reduce to a basis of $\Omega^1_{K(X)/K}$ then
			$$\Omega_{K(X)/K}^1 = K(X) + K(X)\eta_1 + \cdots + K(X) \eta_n,$$
			and the sum is direct. 
		\end{enumerate}
	\end{lemma} 
	\begin{proof}
		The proofs in both of these cases are the same. 
		Let $\eta$ be a KS-form. 
		Let $\overline{\eta}$ be its reduction to a Kahler form (and write $\overline{\eta}_i$ for the reductions of $\eta_i$).
		By hypothesis we may write 
		$$ \overline{\eta} = f_1 \overline{\eta}_1 + \cdots + f_n\overline{\eta}_n. $$
		We now have 
		$$ \lambda(\eta - f_1 \eta_1 -\cdots - f_n \eta_n) = 0 $$
		by exactness $\eta - f_1 \eta_1 -\cdots - f_n \eta_n = \iota(g)$ which shows $\eta$ has the approprate form.
		
		By the property that the $\overline{\eta}_i$'s are a basis for the usual forms we get uniqueness of the $f$'s. 
		By injectivity of $\iota$ we get uniqueness of $g$.
	\end{proof}
	

		As each $\eta\in \Omega^{\tau}_{X/K}$ is an element of $\OO_{\tau X}$ we can view it as a map $\tau X \to \mathbb{A}^1$.
		This means we have a pairing 
		$$ \Omega^{\tau}_{X/K} \times \tau X \to \AA^1.$$
		Let $U$ be an open subset of $X$. 
		For $x\in X(\Khat)$ and $\eta \in \Omega^{\tau}_{X/K}(U)$ with $x \in U(\Khat)$ we get a well-defined elements $\eta( \exp_1(x)) \in \widehat{K}$ and the equation $ \eta(\exp_1(x))=0$ defines a Kolchin closed subset of $U(\Khat)$. The details of this approach will be developed in Section \ref{tauforms}.

	\subsection{Splittings and prolongations}\label{S:splits-are-trivs}
	Let $X$ to be a scheme over a differential ring $(K,\delta)$. 
	Assume that $X$ is smooth over $\Spec K$. 
	Define $\Scal_{X/K}$ to be the sheaf of locally splittings of the exact sequence 
	$$ 0 \to \OO_X \xrightarrow{\iota} \Omega^{\tau}_{X/K} \xrightarrow{\lambda} \Omega^1_{X/K} \to 0.$$
	That is, for each $U \subset X$ we let $\Scal_{X/K}(U)$ be the collection of morphisms of $\OO_X(U)$-modules $\sigma: \Omega_{X/K}(U) \to \Omega_{X/K}^{\tau}(U)$ such that $\lambda \sigma = \id$ where $\id$ is the identity morphism on $\Omega_{X/K}(U)$.
	
	\begin{lemma}\label{L:equivalence-torsors}
		\begin{enumerate}
			\item The sheaf $\Scal_{X/K}$ is a torsor under $T_{X/K}$.
			\item We have $\tau X \cong \Scal_{X/K}$ as torsors under $T_{X/K}$
		\end{enumerate}
	\end{lemma}
	\begin{proof}
		Let $\sigma_1$ and $\sigma_2$ be two splittings. 
		Since $\sigma_1 - \sigma_2$ has an image in $\OO(U)$ it follows that 
		$ \sigma_1 - \sigma_2 \in \Hom^*(\Omega(U),\OO(U)) \cong \Der(\OO(U)/K). $
		
		Conversely for $\partial \in \Der(\OO(U)/K)$ and $\sigma$ a splitting, the map $ f_{\partial} + \sigma $
		defines another splitting.
		Here is how we define $f_{\partial}$: for $\partial \in \Der(\OO(U)/K)$ the map $f_{\partial}: \Omega^1_{U/K}(U) \to \Omega^{\tau}_{U/K}(U)$ given by 
		$ f_{\partial}(dg) = \iota(\partial g), $
		for $g \in \OO(U)$.
		
		We now prove the second assertion.
		Since any morphism of torsors is automatically an isomorphism we just need to exhibit a morphism.
		Local sections of $\pi:\tau X \to X$ over a set $U$ give derivations $\delta_1: \OO(U) \to \OO(U)$ extending the derivation $\delta:K \to K$.
		The map $\tau X(U) \to \Scal_{X/K}(U)$ is given by $\delta_1 \mapsto \sigma_{\delta_1}$ where 
		$$\sigma_{\delta_1}(da) = \iota(\delta_1(a))-d^{\tau}(a), \quad a \in \OO(U).$$
		Note that $\sigma_{\delta_1}(da) - \sigma_{\delta_2}(da) = \iota(\delta_1(a)-\delta_2(a))$ which gives the torsor structure on derivations after applying $\iota^{-1}$.
		This proves the result. \end{proof}
	
	We remark that one can write down the inverse map concretely.
	If $\sigma: \Omega^1_{A/K} \to \Omega^{\tau}_{A/K}$ is a splitting we define $\delta_{\sigma} \in \tau X(U)$ by 
	$$ \delta_{\sigma}(x)=\iota^{-1}(\sigma(dx) -d^{\tau}x ), \quad x \in \OO(U). $$
	By inspection, $\delta_{\sigma}$ lifts the derivation on $K$.
	Also, and the difference of two such derivations is a $K$-linear derivation.
	

	\section{ODEs coming from forms}\label{tauforms}
	
	\subsection{$D$-Schemes, $D$-varieties, and varieties with forms} 
	An introduction to $D$-Schemes can be found in \cite[section 3]{Buium1993}. We will quickly review the notation here as well as the slightly different notions used by \cite{Hrushovski2003} and \cite{Rosen2007}, as well as an equivalence between several of the notions, since we will use results from each of these sources. This will involve introducing several categories, all of which encode essentially the same information, but each of which offers advantages in certain circumstances. 
	
	Let $K$ be a ring with derivation $\delta$.
	\begin{definition}
		A $D$-scheme $(V,\delta_V)$ is a scheme $V/K$ together with a derivation on the structure sheaf lifting the derivation on the base.
		Morphisms of $D$-schemes $(V,\delta_V) \to (W,\delta_W)$ are morphisms of $K$-schemes $V \to W$ which are equivariant with respect to the derivation on the structure sheaf. 
		We will denote this category by $\DSch_K$.
	\end{definition}
	
	When the underlying scheme of a $D$-scheme is a variety we will call it a \emph{$D$-variety.} Since local sections of $\tau V$ represent extensions of derivations, one can replace $\delta_V$ in the definition with its corresponding section $s_V: V \to \tau V$. 
	That is if $\delta_V:\OO_V \to \pi_*\OO_{\tau V}$ is the universal derivation extending the derivation on $K$ and $s_V: V \to \tau(V)$ is a section then we have the relation
	$$ \delta_V  = s^{*}_V \circ \delta_{univ},$$
	i.e. the derivation $\delta_V$ is obtained by first taking the universal derivation and then looking at the associated map of sheaves of rings.
	
	\begin{definition}
		Define the category $\DSch^*_K$ of \emph{alternate $D$-schemes} to be the category whose objects are pairs $(V,s)$ where $V$ is a $K$-scheme and $s:V\to \tau V$ is a section of the first jet space. 
		In this category morphisms will be given by morphisms of varieties $f:V \rightarrow W$ such that the diagram commutes: 
		$$\xymatrix{ V \ar[d]_s  \ar[r]^f & W \ar[d]_t\\
			\tau V \ar[r]_{\tau f} & \tau W}.$$
	\end{definition}
	
	From the above discussion: 
	
	\begin{lemma}
		The categories $\DSch_K$ and $\DSch_K^*$ are equivalent. 
	\end{lemma}
	
	This means pairs $(V,\delta_V)$ are equivalent to pairs $(V,s)$ and morphisms of $D$-schemes are equivalent to morphisms of alternate $D$-schemes. 
	From now on we will conflate the categories $\DSch_K$ and $\DSch_K^*$ without comment and move freely between sections $s_V$ and derivations $\delta_V$. The above discussion is quite similar to that of \cite[pages 4271-4272]{Hrushovski2003}, where additionally, the above categories are also considered modulo sets of smaller dimension. We've included the discussion for clarity, since various model-theoretic sources (e.g. \cite{Hrushovski2003}) primarily use the category $\DSch_K^*$ (or even that category modulo sets of smaller dimension) while the results of Buium we cite (e.g. \cite{Buium1994a}) use the category $\DSch_K $. 
	
	To every $D$-variety $(V,s)$ over the differential ring $(K,\delta)$ we associated the differential scheme $(V,s)^{\sharp}$ which associates to any differential $K$-algebra $L$ the set 
	$$ (V,s)^{\sharp}(L) = \lbrace x \in V(L): \exp_1(x) = s(x) \rbrace. $$ 
	These are the solution sets given by order one differential equations we are interested in. 
	For a proof of the next Lemma, see \cite[Introduction]{Hrushovski2003}.

	\begin{lemma}
		Let $(K,\delta)$ be a differential field.
		Every irreducible differential algebraic $K$-variety of finite absolute dimension is birational to some $(V,s)^{\sharp}$ for some $(V,s) \in \DSch_K$.
	\end{lemma}

	The next Lemma is well-known. See \cite{Marker1995}, for instance, where the more general fact that the absolute dimension bounds the Morley rank is proved. We include a proof, because in this special case, one can argue in a completely elementary manner. 
	
	\begin{lemma}\label{L:strong-minimality}
		If $C$ is an irreducible curve, and $(C,s)$ a $D$-variety then $(C,s)^{\sharp}$ is strongly minimal. 
	\end{lemma}
	\begin{proof}
		The problem is local and we can write $C = \Spec K[x]/(f)$ where $x$ and $f$ are tuples. 
		By \cite[Chapter 3, 3.9]{Buium1994} the scheme corresponding to $(C,s)^{\sharp}$ in $C^{\infty}$ is given by $\Spec K[x,x']/\langle f, f',x'=a(x)\rangle$ where the last equation is given by $(x_1',\ldots,x_n') = (a_1(x),\ldots,a_n(x))$ with $a_j(x)\in K[x]$.
		The $a_j(x)$ here represent the map $C\to \tau C$.
		The scheme $(C,s)^{\sharp} \subset C^{\infty}$ is clearly one dimensional and hence all of its subschemes (and hence $\delta$-subschemes) are either finite or cofinite. 
	\end{proof}
	
	In order to keep track of KS-forms together with a variety we define a new category. 
	\begin{definition}
		Let $(K,\delta)$ be a differential ring. 
		We define the category $\Var^{\sharp}_K$ of \emph{marked varieties} whose objects consist of pairs $(V,\eta)$ where $V$ is a $K$-variety and $\eta\in \Omega^{\tau}_{K(V)/K}$ is a rational KS-form. 
		A morphism in this category $(V_1,\eta_1) \to (V_2,\eta_2)$ will be a rational map $\varphi:V_1 \to V_2$ of $K$-varieties such that $\varphi^*\eta_2 = \eta_1.$
	\end{definition}

	In this section suppose that $C$ is a curve over a differential field $(K,\delta)$ of characteristic zero. 
	Let $\eta \in \Omega^{\tau}_{C/K}(U)$ where $U \subset C$ is some open set. 
	Without loss of generality, we can assume that $\eta$ is of the form $ud^{\tau}x + g$ for some local parameter $x$ and $u,g\in \OO_C(U)$ with $u$ a unit.
	We will show how to associate to $\eta$ a derivation $\delta_{\eta}.$
	
	We use $\eta$ to get a map $q:\OO^{[1]}(U) \to \OO(U) = \OO^{[1]}(U)/\langle \eta \rangle$ and hence defining a derivation prolonging the derivation on the base via the universal property of jet spaces. 
	The derivation is
	$$ \delta_{\eta} = q \circ d^{\tau}.$$
	This gives us a $D$-scheme $(X,\delta_{\eta})$ 
	
	\begin{remark}
		Alternatively, for $U \subset C$ open one may view $\eta \in \Omega^{\tau}_{C/K}(U)\setminus \OO(U)$ as inducing a splitting of $\Lambda: \Omega^{\tau}_{C/K}(U) \to \Omega_{C/K}^1(U).$
		By Section~\ref{S:splits-are-trivs} we know that splittings are the same as trivializations of the first jet space which are the same as derivations. 
	\end{remark}
	
	Conversely, suppose that instead of starting with an element of $\Omega^{\tau}_{K(C)/K}$ we are given a derivation and hence by Section~\ref{S:splits-are-trivs} a local splitting $\sigma_{U}$ of $\Lambda:\Omega^{\tau}_{C/K}(U) \to \Omega^1_{C/K}(U)$. 
	Then for each $\omega \in H^0(U,\Omega^1_{C/K})$ we may construct $\sigma_U(\omega)$ a section and a pair $(U,\sigma_U(\omega))$. 
	
	The constructions are only inverse to each other where $\sigma_U(\omega)$ is a local generator, meaning we need to remove the zero locus of $\omega$. 
	Also, if $\omega' \in H^0(U,\Omega^1_{C/K})$ is another local generator, then there exists some $f \in \OO_C(U)$ such that $\omega' = f \omega$. 
	This in turn implies that 
	$$f \sigma_U(\omega) = \sigma_U(\omega').$$
	It is our aim in the next sections to prove a converse of this.
	First, to simplify notation on open sets, for $\eta \in \Omega_{K(C)/K}^{\tau}$ we will let 
	$$ (C,\eta)^{\sharp} = (U,\eta)^{\sharp},$$
	where $U$ is the maximal domain of definition for $\eta$.

	\subsection{Forms: new and old}
	In this section, we will assume $(K, \delta)$ is a differential field of characteristic zero with the property that there is some $t \in K$ such that $\delta(t)=1.$ 
	Also curves are assumed to be smooth and projective. 
	
	\begin{definition} 
		Let $C/K$ be a curve over a differential field $K$.
		\begin{enumerate}
			\item A differential form $\omega \in \Omega^1_{K(C)/K}$ is \emph{old} if there exists a morphism of curves $g:C\to C'$ of degree at least two and some $\omega' \in \Omega^1_{K(C)/K}$ such that $ \omega = g^*\omega'. $
			A differential form which is not old will be called \emph{new}.
			
			\item A KS-form $\omega \in \Omega^{\tau}_{K(C)/K}$ is \emph{old} if there exists morphism $g : C \rightarrow C'$ of degree at least two and some $\omega' \in \Omega^{\tau}_{K(C')/K}$ such that $\omega = g^{ *} \omega',$ where $\eta\in \Omega^{\tau}_{K(W)/K}$. 
			A KS-form which is not old will be called  \emph{new}.
		\end{enumerate}
		
		\begin{remark}
			In \cite{Hrushovski2003} the terminology \emph{essential} is used instead \emph{new}. The term \emph{new} is already in use in the modular forms literature, hence we adopt this terminology here. 
		\end{remark}
		
	\end{definition} 
	
	It will be important to consider rational equivalence classes of rational KS-forms.
	\begin{definition}
		For $\eta_1,\eta_2 \in \Omega^{\tau}_{K(V)/K}$ we will write say $\eta_1$ and $\eta_2$ are \emph{rationally equivalent} and write $\eta_1\sim \eta_2$ if and only if $\exists f \in K(V), \ f \eta_2 = \eta_2.$
	\end{definition}
	For two differential varieties $\Sigma$ and $\Gamma$ we will write $\Sigma \approx \Gamma$ if and only if the symmetric difference of $\Sigma(\Khat)$ and $\Gamma(\Khat)$ is finite.
	Later we will show that if $\eta \sim \eta'$ then $(C,\eta)^{\sharp}$ and $(C,\eta')^{\sharp}$ are birational  (see Lemma~\ref{L:simsimsim}).
	
	For $\eta\in \Omega^{\tau}_{K(V)/K}$ we let $[\eta]$ denote the rational equivalence class of forms, i.e. $$[\eta] = \lbrace \eta' \in \Omega^{\tau}_{K(V)/K} : \eta \sim \eta' \rbrace. $$ 
	We will also make use of the notion of global and new rational equivalence classes of KS-forms.
	\begin{definition}
		We call $[\omega]$ \emph{global} if $[\omega]=[\omega_1]$ for some $\omega _1 \in H^0(X,\Omega ^\tau_{X/K}).$ We will call a class $[\omega]$ \emph{new} if $\forall \omega _1\in [\omega]$,  $\omega_1$ is new. 
	\end{definition}
	
	A main question left open by Rosen is whether there exist new rational equivalence classes (see \cite[Remark 4.12]{Rosen2007}). We will resolve this question in Section \ref{backtothejacobian}. Of course, the main motivation for showing the existance of these new forms is (as we will show) they correspond to new nonorthogonality classes of strongly minimal differential equations. In the special case that the equations are autonomous, \cite{Hrushovski2003} gave many examples of nonorthogonality classes of strongly minimal differential equations. This was accomplished by working in a simpler formalism with the sheaf of differential forms (which is only feasible when the equations are assumed to be autonomous) and finding differential forms which can not be the pullback of forms on lower genus curves. For details, see Appendix \ref{HrItapp}.

	
	\subsection{ODEs and KS-forms} \label{S:rosen-redux}
	Let $\omega \in \Omega^\tau_{K(C)/K}$ be a nonzero KS-form. 
	
	\begin{lemma} \label{L:equivforms}
		There is a faithful functor from the category of $D$-varieties to the category of varieties equipped with a KS-form modulo $\approx$ via associating $s$ with any non-fiber constant KS-form $\omega$ such that $\omega (s)$ is zero almost everywhere. 
	\end{lemma}
	
	The choice of $\omega$ is not uniquely determined, but it is easy to see that any two such choices $\omega_1$ and $\omega _2$ have the property that the associated systems are birational, i.e.
	$$(C, \omega _1)^ \sharp \approx (C, \omega _2)^ \sharp $$ 

	\begin{lemma}\label{L:simsimsim}
		Let $C$ be a curve and $\eta_1,\eta_2 \in \Omega^{\tau}_{K(C)/K}\setminus K(C)$. We have
		\begin{equation}
		(C,\eta_1)^{\sharp} \approx (C,\eta_2)^{\sharp} \iff \exists f \in K(C) \ \ \eta_2 = f \eta_1
		\end{equation}
	\end{lemma}
	\begin{proof}
		Suppose $(C,\eta_1)^{\sharp} \approx (C,\eta_2)^{\sharp}$. 
		By Lemma~\ref{L:basis} we have $\eta_2 = f \eta_1 + g$ for some $f,g\in K(C)$.
		At each $v \in \tau C(\Khat)$ with $\pi(v)=x \in C(\Khat)$ we have 
		$$ \eta_2(v) = f(x) \eta_1(v) + g(x),$$
		whenever this makes sense. 
		Here we used linearity of the pairing.
		Since $\eta_1(v)$ and $\eta_2(v)$ almost have the same zero set, we must have $g=0$ and hence $\eta_2(v) = f(x)\eta_1(v)$.
		This implies that $\eta_2 = f \eta_2$.
		
		The converse follows from the fact that $f$ only finitely zeros and poles.
	\end{proof}
	
	We now have a well-defined procedure for associating a a strongly minimal sets (modulo equivalence up to finitely many points) living on $C$ to a rational equivalence class of KS-forms in $\Omega_{K(C)/K}^{\tau}$.
	
	\begin{lemma}[{\cite[starting on page 12]{Rosen2007}}] \label{L:behavior} 
		Let $g: C \rightarrow C'$ be a non-constant morphism of smooth projective curves over a differential field $K$. 
		Let $\omega \in \Omega^{\tau}_{K(C)/K}$ and $\omega'\in \Omega^{\tau}_{K(C)/K}$. 
		\begin{enumerate}
			\item  
			$\omega := g ^ * \omega ' \implies g^{-1} ( (C', \omega ')^ \sharp ) = (C , \omega  )^ \sharp.$ 
			\item \label{I:pully} 
			$g( (C, \omega )^ \sharp ) \approx (C', \omega ')^ \sharp  \implies [\omega]=[  g^* \omega '].$ 
			\item $g((C,\omega )^\sharp )= (C',\omega ')^\sharp \mbox{ and } \omega' \mbox{ global } \implies [\omega] \mbox{ global }$
		\end{enumerate}
	\end{lemma} 
	\begin{proof} 
		\begin{enumerate}
			\item For $c \in C(\Khat)$ we have $g^*\omega'(\exp_1(c)) =\omega'(\tau_g(\exp(c))) = \omega'(\exp_1(g(c)),$ so
			$c \in (C,g^*\omega')^{\sharp}(\Khat)$ if and only ig $g(c) \in (C',\omega')^{\sharp}(\Khat)$.
			\item This is a special case of Lemma~\ref{L:simsimsim}.
			\item A non-constant rational map from a complete curve $C$ necessarily a morphism and surjective. 
			Since $\omega '$ is global and $\omega \sim g^* \omega '$ we have that $[g^*\omega']$ is global by definition.
		\end{enumerate}
	\end{proof} 

	\section{A Criterion for geometrically disintegrated systems}\label{S:strictly-disintegrated}
	
	
	\begin{lemma}\label{onemin} 
		Let $C$ be a smooth projective curve over a differential field $K$. 
		Suppose that $\omega \in \Omega_{K(C)/K}^{\tau}$. 
		The class $[\omega]$ is new if and only if every $K$-definable equivalence relation on $(C, \omega )^\sharp$ with generically finite classes has the property that all but finitely many classes have size one.  
	\end{lemma} 
	
	\begin{proof} 
		The proof of the forward direction is by contrapositive.
		Suppose that you have a $K$-definable equivalence relation which has generically finite classes and has infinitely many classes with class size larger than one.
		By elimination of imaginaries (Lemma~\ref{L:EIsoup}), there is a map to another curve $C'$ and by Lemma~\ref{L:equivforms} there exists a rational KS-form $\omega'$ and $g:C \rightarrow C'$ with degree larger than one such that $g((C, \omega )^ \sharp) \approx (C' \omega ') ^ \sharp$. 
		Take the pullback of $\omega '$ and apply Lemma~\ref{L:behavior} item \eqref{I:pully}. 
		Hence the form is old.
		
		For the converse, suppose that $\omega$ is the pullback of $\omega ' $ on $C'$ along $g: C \rightarrow C'$ of degree bigger than one. 
		Then generically, the inverse image of a point of $C'$ consists of more than one point on $C$. 
		Let $E$ be the equivalence relation on $C$ induced by being in the same fiber of $g$. 
	\end{proof} 

	\begin{theorem}[Finiteness and Disintegration]\label{thmrosen}
		Let $K$ be a finitely generated differential field. Let $C/K$ be a smooth projective curve of genus one or more. 
		Let $\omega \in \Omega_{K(C)/K}^{\tau}$.
		If $[\omega] $ is new then $|(C,\omega)^{\sharp}(K)| < \infty$.
		Furthermore, if $C' = (C,\omega )^{\#} \setminus (C,\omega)^{\#}(K)$ then $(C',\omega)^{\sharp}(\Khat)$ is  disintegrated and isolates a type.
	\end{theorem} 
	\begin{proof} 
		By Lemma \ref{onemin} we must have that $(C, \omega )^ \sharp$ is weakly orthogonal to the constants. If $(C,\omega)^\sharp$ is nonorthogonal to the constants then because $(C,\omega)^\sharp$ is strongly minimal, then some finite quotient of $(C,\omega)^\sharp$ over $K$ is internal to the constants. Then this quotient is represented by $(C_1, \omega_1)^ \sharp$ over $K$ by Lemma~\ref{L:equivforms}. Then the action of the binding group of $(C_1, \omega_1)^ \sharp$ relative to the constants is isomorphic to the transitive action of an infinite algebraic group acting birationally on the curve $C_1$. This curve must be of genus $0$ or $1$, since curves of genus larger than one have finite groups of birational transformations. Then the quotient map $f: C \rightarrow C_1$, by Riemann-Roch, must have degree larger than one.
		Now, by \cite[Theorem 6.2]{freitag2017finiteness}, it follows that $(C, \omega )^\sharp $ is geometrically trivial, $\aleph _0$-categorical, and 
		$$\#(C, \omega )^ \sharp(K)<\infty.$$ 
		Now, every remaining element of $(C, \omega )^ \sharp(\Khat)  \setminus (C,\omega)^{\sharp}(K)$ has the same type over $K$, and if two such elements are algebraically dependent by some formula $\phi (x,y)$, then $\phi (x,y) \vee \phi (y,x)$ defines an equivalence relation with generically finite classes on $(C, \omega )^ \sharp(\Khat)  \setminus (C,\omega)^{\sharp}(K)$. 
		The property that every class has size larger than one gives a contradiction. 
	\end{proof}

	\begin{cor}
		Suppose $C$ is a complete nonsingular curve of genus at least 2. 
		Let $\omega \in \Omega_{K(C)/K}^{\tau}$ such that $[\omega]$ is new. For all curves $C'$ and forms $\omega' \in \Omega^{\tau}_{K(C)/K}$ there is at most one map $g: C' \rightarrow C$ such that $\omega ' = g^* \omega$. 
	\end{cor}
	
	\begin{proof} Suppose that there are two such maps, $g_1$ and $g_2$. Now, one can generically define a (nonidentity) generically finite-to-finite correspondence between $(C,\omega)^{\sharp}$ and itself, $x \mapsto g_2 (g_1^{-1} (x)).$ This is impossible by Theorem \ref{thmrosen}. 
	\end{proof} 

	\section{Global new forms} \label{backtothejacobian}
	
	In this section we prove the existence of new KS-forms, which are global on nonisotrivial curves. 
	\subsection{Global forms on the curve and global forms on the Jacobian}\label{Jacobian}
	The following answers a question posed in \cite{Rosen2007}.
	\begin{theorem}\label{T:global-forms}
		Let $C$ be a smooth projective curve over $K$ of genus $g\geq 2$, and let $J$ its Jacobian. Then the Abel-Jacobi embedding $j \colon C \to J$ determines an isomorphism 
		$$H^0(C, \Omega^{\tau}_C) \cong H^0(J, \Omega^{\tau}_J).$$
	\end{theorem}
	
	\begin{proof}[Proof of Theorem~\ref{Jacobian}]
		This proof was worked out with Aaron Royer while we were at MSRI.
  Let $j \colon C \to J$ be an Abel-Jacobi map. 
		We have morphisms $j^* \Omega_J  \to \Omega_C,$ and $j^* \OO_J[1] \to \OO_C[1]$ of $\OO_C$-modules induce corresponding morphisms $D^b(C)$ which give the commuting square
		\begin{equation}\label{hom square}
		\begin{tikzpicture}
		\matrix (m) [matrix of math nodes,row sep=3em,column sep=4em,minimum width=2em]
		{
			\Hom_{D^b(C)}(j^*\Omega_J,j^*\OO_J[1]) & \Hom_{D^b(C)}(\Omega_C,j^*\OO_J[1]) \\
			\Hom_{D^b(C)}(j^*\Omega_J,\OO_C[1]) & \Hom_{D^b(C)}(\Omega_C,\OO_C[1]). \\};
		\path[-stealth]
		(m-1-1) 
		edge node [left] {$\alpha$} (m-2-1)
		(m-1-2) 
		edge node [right] {$\gamma$} (m-2-2)
		edge node [above] {$\delta$} (m-1-1)
		(m-2-2) 
		edge node [below] {$\beta$} (m-2-1);
		\end{tikzpicture}
		\end{equation}
		The maps $\alpha$ and $\gamma$ are isomorphisms since the natural map $\eps: j^*\OO_J \to \OO_C$ is an isomorphism of $\OO_C$-modules.
		For all varieties $X/K$ we will abuse notation letting $\KS_X(D) \in \Hom(\Omega_X, \OO_C[1])$ denote the class of $\KS_{X/S}(D) \in H^1(X,T_X)$ under the isomorphism
		$$\Hom(\Omega_{X/S}, \OO_X[1]) \cong \Ext^1(\Omega_{X/S},\OO_X) \cong H^1(X,T_{X/S}).$$
		By Theorem~\ref{T:functorality-of-KS} we have
		$$ j^* \KS_J(D) = \alpha\circ \eps^{-1}( \KS_C(D) ). $$
		By functorality we have the commutatitive diagram
		$$\xymatrix{
			0 \ar[r] & \OO_C \ar[r] & \Omega_C^{\tau} \ar[r] & \Omega_C \\
			0 \ar[r] & \OO_C \ar[u] \ar[r] & j^* \Omega_J^{\tau} \ar[r] & j^*\Omega_J \ar[u]
		}
		$$
		which we view as a diagram in the bounded derived category of quasicoherent sheaves on $C$, $D^b(C)$.
		By \cite[Definition 10.2.1; TR3]{Weibel1995}, there exists a map $\nu$ filling in the diagram
		$$ 
		\xymatrix{
			0 \ar[r] & \OO_C \ar[r] & \Omega_C^{\tau} \ar[r] & \Omega_C \\
			0 \ar[r] & \OO_C \ar[u] \ar[r] & j^* \Omega_J^{\tau} \ar[r] \ar[u]^{\nu} & j^*\Omega_J \ar[u]
		}.
		$$
		By \cite[Definition 10.2.1; TR2]{Weibel1995} we may rotate the diagram to get:
		$$ 
		\xymatrix{
			0 \ar[r] & \OO_C \ar[r] & \Omega_C^{\tau} \ar[r] & \Omega_C \ar[rr]_-{\KS_{C/K}(D)} & & \OO_C[1]\\
			0 \ar[r] & \OO_C \ar[u] \ar[r] & j^* \Omega_J^{\tau} \ar[u]^{\nu} \ar[r] & j^*\Omega_J \ar[u] \ar[rr]_-{j^*\KS_{J/K}(D)} & & \OO_C[1] \ar[u]
		}
		$$
		and this map is identified with the $\Ext^1$ class of the sequence.
		Applying $\Hom(\OO_C,-)$ to the above diagram, using  $\Ext^1(A,B)\cong\Hom(A,B[1])$ in derived categories, and using $H^1(X,F) \cong \Ext^1(\OO_X,F)$ for any scheme $X$ and a quasicoherent sheaf $F$ on $X$ we get
		$$ 
		\xymatrix{
			0 \ar[r] & H^0(C,\OO_C) \ar[r] & H^0(C,\Omega_C^{\tau}) \ar[r] & H^0(C,\Omega_C ) \ar[rr]_{\KS_C(D)} & & H^1(C,\OO_C)\\
			0 \ar[r] & H^0(C,\OO_C) \ar[u] \ar[r] & H^0(C,j^* \Omega_J^{\tau}) \ar[u]^{\nu} \ar[r] & H^0(C,j^*\Omega_J) \ar[u] \ar[rr]_{\KS_J(D)} & & H^1(C,j^*\OO_J) \ar[u]
		},
		$$
		where we have abused notation letting $v$ denote $\Hom(\OO_C,v)$. 
		The $5$-lemma then proves 
		$$ H^0(C, \Omega^{\tau}_C) \cong H^0(C, j^*\Omega_J^{\tau}).$$ 
		Using the diagram induced by pullbacks
		$$ 
		\xymatrix{
			0  \ar[r] & H^0(C,\OO_C)  \ar[r] & H^0(C,j^* \Omega_J^{\tau}) \ar[r] & H^0(C,j^*\Omega_J)  \ar[rr]_{j^*\KS_J(D)} & & H^1(C,j^*\OO_J) \\
			0 \ar[r] & H^0(J,\OO_J)\ar[u] \ar[r] & H^0(J,\Omega_J^{\tau}) \ar[r] \ar[u] & H^0(J,\Omega_J ) \ar[u] \ar[rr]_{\KS_J(D)} & & H^1(J,\OO_J) \ar[u] \\
		}.
		$$
		Again, by the 5-lemma we get
		$$H^0(C,j^*\Omega_J^{\tau}) \cong H^0(J,\Omega_J^{\tau})$$ 
		from another application of the $5$-lemma.
	\end{proof}
	
	\subsection{Hrushovski-Itai type criteria for global new forms}
	
	In this section we apply Theorem~\ref{T:global-forms} to give a criterion for the existence of new forms. The criterion here is somewhat similar to that of Hrushovski and Itai - see Appendix  .

	

	\begin{theorem} 
		\begin{enumerate}
			\item \label{I:hrush-proof} If $A$ is an abelian variety then the collection of old forms $N \subset H^0(A,\Omega^{\tau}_{A/K})$ is contained in a countable union proper subspaces if $h^0(A,\Omega^{\tau}_{C/K})\geq 2$.   
			\item Let $C/K$ be a smooth projective curve.
			The collection of $N \subset H^0(C, \Omega^1_{C/K})$ of old forms is contained in the countable union of proper subspaces provided $h^0(C,\Omega^1_{C/K})\geq 2$ is of dimension  bigger than two. 
		\end{enumerate}
	\end{theorem}
 
	\begin{proof}
		
		\begin{enumerate}
			\item If $\eta$ is a global 1-form of $C$ which is old there exists a degree greater than one map of smooth projective curves $f: C \to C'$ such that $\eta= f^* \xi$. 
			Recall that global one forms on $C$ are exactly the invariant differential forms on $J$ where $J$ is the Jacobian of $C$. 
			Using the Abel-Jacobi map and functorality we can write 
			$$ \eta = j^*\Jac(f)^* \xi, $$
			where $\Jac(f): J \to J'$ is the induced map on Jacobians and $j: C \to J$ is the Abel-Jacobi map.
			In particular for all old $\eta$ there exists some $A \subsetneq J$ such that
			$$\eta \in V_{A} = \lbrace \pi_A^* \omega \colon \omega \in \Omega_{(J/A)/K}^1(J/A) \rbrace = \lbrace s \in \Omega^1(J) : s \vert_A = 0 \rbrace.$$ 
			Then we have 
			$$ N = \bigcup \lbrace V_A: A \subset J \mbox{ and } A \neq J \rbrace.  $$
			Since there are only countably many Abelian subvarities (see e.g. \cite[P3,page 180]{Bouscaren1997}) we are done. 
			\item The proof is very similar to that of part \ref{I:hrush-proof}.
		\end{enumerate}
	\end{proof}

	Note that in particular if $C$ is a curve with a simple Jacobian then every nonzero global section of $H^0(C,\Omega^{\tau}_{C/K})$ is new.

	\subsection{Global new forms, the Manin kernel, and cupping with Kodaira-Spencer}

Our terminology in the next definition follows \cite[section 5]{Rosen2007a}; The same invariant is also used and defined in \cite[Introduction]{Buium1993} where it is called the $\delta$-rank.
 
	\begin{definition}
		Let $X/K$ be an abelian variety of projective curve over a differential field $(K,\delta)$. 
		The \emph{Kodaira-Spencer rank} $\rk(\KS_{X/K}(\delta))$ is the rank of  the cup-product-with-Kodaira-Spencer map $H^0(X,\Omega_{X/K}^1) \to H^1(X,\OO_X)$.
	\end{definition}

	We also make use of the notion of a Manin kernel. 
	\begin{definition}
		Let $A/K$ be an abelian variety over a differential field $(K,\delta)$. 
		The \emph{Manin kernel} $A^{\sharp}$ is the differential algebraic subvariety of $A$ defined to be the $\delta$-closure of the torsion points.
	\end{definition}
	The Kodaira-Spencer rank $\rk(\KS_A(\delta))$, number of global KS-forms $h^0(X,\Omega_{A/K}^{\tau})$, and absolute dimension of the Manin Kernel $a(A^{\sharp}) $ all determine each other once we know $g$.
	\begin{theorem}\label{T:manin-kernels-and-tau-forms}
		\begin{enumerate}
			\item \label{E:buium-book} Let $A/K$ be an abelian variety over a differential field $(K,\delta)$.
			We have
			$$a(A^{\sharp}) = g+\rk(\KS_X(\delta)).$$
			\item \label{E:manin-kernel-tau-forms} Let $A/K$ be an abelian variety. Then
			$$h^0(\Omega^{\tau}_{A/K}) =g+1-\rk(\KS_X(\delta)).$$
		\end{enumerate}
	\end{theorem}
	\begin{proof}	
		The first statement is \cite[V, (3.18) and (3.19)]{Buium1992}.
		For any variety $X/K$ we have an exact sequence of vector bundles
		$$ 0 \to \OO_X \to \Omega_{X/K}^{\tau} \to \Omega_{X/K}^1 \to 0 $$
		where the extension class is the Kodaira-Spencer class $\KS_X(\delta)$. 
		Taking the associated long exact sequence in sheaf cohomology gives
		$$ H^0(\OO_X) \to H^0(\Omega^{\tau}_{X/K}) \to \ker(c: H^0(\Omega_{X/K}^1) \to H^1(\OO_X)), $$
		where the map $c$ is given by $c(\omega) = \omega \smile \KS(\delta)$. 
		We denote the rank of the cup product map $c$ by $\rk(\KS_X(\delta))$.
		This means that 
		\begin{equation}\label{E:rosen}
		h^0(\Omega^{\tau}_{X/K}) =  1+g-\rk(\KS_X(\delta)).
		\end{equation}
	\end{proof}
	
By reasoning similar to the proof of Theorem~\ref{T:manin-kernels-and-tau-forms} we can give a criterion for when a global differential form $\omega$ lifts to a global KS-form.
	
	\begin{theorem}
		An element $\omega \in H^0(X,\Omega_{X/K})$ lifts to an element of $H^0(X,\Omega^{\tau}_{X/K})$ if and only if $\KS_X(\delta)\smile\omega=0$.
	\end{theorem}
	\begin{proof}
		In general given $E$ which can be written as an extension of sheaves of modules 
		$$ 0 \to E' \to E \to E'' \to 0, $$ 
		if one has an explicit description of $H^0(E'')$ one can try to lift these to global sections to $E$ by locally solving for the components of $E'$ that ``correct'' these lifts. 
		It turns out that these corrections give rise to a cohomology class and a map 
		$$ c:H^0(X,E'') \to H^1(X,E').$$
		Here is what this looks like in our application. 
		Let $\omega \in H^0(X,\Omega^1_{X/K})$. 
		Let $(U_i\to X)_{i\in I}$ be a Zariski affine open cover such that for each $U_i$ the sequence
		\begin{equation}\label{E:local-sequence}
		0\to \OO_X(U_i) \to \Omega^{\tau}_{X/K}(U_i) \to \Omega^1_{X/K}(U_i) \to 0
		\end{equation}
		splits.
		This means that for each $U_i$ there exists some $\eta_i \in H^0(U_i,\Omega^{\tau}_{X/K})$ which lifts $\omega\vert_{U_i}$. 
		In order for a global lift $\eta$ to exist the cohomology class 
		$$c(\omega):=[\iota^{-1}(\eta_i - \eta_j)] \in H^1(X,\OO) $$
		must vanish. 
		
		Here is how the Kodaira-Spencer class fits into this picture. 
		Given local splittings $\sigma_i$ of the sequence $\eqref{E:local-sequence}$ one can take $\eta_i = \sigma_i(\omega\vert_{U_i})$ as above. 
		One then finds that $\eta_i - \eta_j = (\sigma_i-\sigma_j)(\omega\vert_{U_{ij}})$ which, using the equivalence between splittings and local derivations (Lemma~\ref{L:equivalence-torsors}) allows us to see that $ c(\omega) = \KS_X(\delta) \smile \omega.$
		This explains precisely why elements in the kernel of $c$ lift to elements of $H^0(X,\Omega^{\tau}_{X/K})$. 
	\end{proof}

Before proceeding to general computations in section~\ref{S:plane-curves} we show how this works in practice.
First, there is the obvious case where the Kodaira-Spencer class is zero; here the Kodaira-Spencer class detects  isotriviality. 
The example we give immediately following the next Lemma is the interesting example for Picard curves.

	\begin{lemma}\label{L:descent}
	Let $X$ be a smooth projective variety  over a differential field $(K,\delta)$ of characteristic zero which is algebraically closed. 
	The following are equivalent.
	\begin{enumerate}
		\item \label{I:D-var} $X$ admits the structure of a global $D$-variety. (Equivalently $\KS_{X/K}(\delta)=0$.)
		\item \label{I:descent} $X$ descends to the constants.
	\end{enumerate}
	\noindent If $X$ is a curve then these are also equivalent to 
	\begin{enumerate}[resume]
		\item \label{I:maximal-forms}   $h^0(X,\Omega^{\tau}_{X/K})=g+1$.
	\end{enumerate}
\end{lemma}
\begin{proof}
	The fact that \eqref{I:D-var} and \eqref{I:descent} are equivalent follows from the Buium-Ehresmann Theorem \cite[pg 65, Prop 2.8]{Buium1994a}.
	
	
	The statement $\eqref{I:descent} \iff \eqref{I:maximal-forms}$ follows from Theorem~\ref{T:manin-kernels-and-tau-forms} where we showed $h^0(X,\Omega_{X/K}^{\tau})=g+1-\rk(\KS_{X/K}(\delta))$.
\end{proof}

As stated previously we now turn to Picard curves where we can show that there exist nontrivial global KS-forms and hence interesting strictly disintegrated order one strongly minimal differential systems.
We will show there exists some $C$ and some $\eta \in H^0(C,\Omega_{C/K})$ with $$\KS_{C/K}(\delta) \smile \eta =0,$$ so that there is some $\widetilde{\eta}$ in $H^0(C,\Omega_{C/K}^{\tau})$ lifting this particular form.

\begin{example}\label{E:picard}
	Consider the Picard curve $C$ given by $y^3=g(x)$ where $$g(x) = x^4+cx+1,$$ where $c$ is some unspecified differential variable with derivative $\delta(c)=c'$ and $c'$ and indeterminate. 
	We will justify these computations in a moment in \S \ref{S:picard-curves}, but just know that this has a smooth projective model as a plane curve with an open cover $C=U \cup V$ where in \v{C}ech cohomology for this cover we can compute  
		$$H^0(C,\Omega)= K \cdot \dfrac{dx}{3y^2} + K \cdot x\dfrac{dx}{3y^2}+ K \cdot y\dfrac{dx}{3y^2}, \quad H^1(C,\OO) = K\cdot [x^2/y]+K \cdot [x^3/y] + K \cdot [x^3/y^2].$$
	The sums here are direct. 
	We can explicitly compute the Kodaira-Spencer class using $D_1 - D_2 \in T_{C/K}(U\cap V)$ for some $D_1\in T_{C}(U)$ and $D_2 \in T_{C}(V)$ prolonging the derivation $\delta$ on the base.
	Using this we compute the cup-product-with-Kodaira-Spencer map
	$H^0(C,\Omega_{C/K}) \to H^1(C,\OO)$. 
	In this particular basis takes the form 
	 $$\begin{pmatrix}
	 0 & 0 & A\\
	 0 & 0 & B \\
	 A & B & 0
	 \end{pmatrix},$$
	 where the constants $A$ and $B$ are given by 
	 \begin{align*}
	 A &= c' \dfrac{11664 c^{8} - 16281 c^{7} - 2160 c^{6} + 34560 c^{5} - 27648 c^{4} + 43776 c^{3} + 65536 c^{2} - 131072 c + 65536}{(27 c^{4} - 256)(432 c^{5} - 459 c^{4} + 176 c^{3} + 768 c^{2} - 768 c + 256)},\\
	B&=c' \dfrac{(3 c - 4)(3888 c^{7} - 4131 c^{6} - 3492 c^{5} - 10992 c^{4} - 44032 c^{3} + 25856 c^{2} - 1024 c - 12288)}{(27 c^{4} - 256)(432 c^{5} - 459 c^{4} + 176 c^{3} + 768 c^{2} - 768 c + 256)}.
	\end{align*}
	The constants above were computed in \texttt{magma}.
	\end{example}
	
	It turns out that these curves are special because they have Jacobians with complex multiplication.
	In Section~\ref{S:applications-to-curves} we apply general results about CM Jacobians to give a second proof that these particular curves have Kodaira-Spencer rank which is non-trivial and non-maximal.
		
	\subsection{Explicit Kodaira-Spencer Computations for Plane Curves}\label{S:plane-curves}
	Consider a smooth plane curve $C \subset \PP^2$ of degree $d$ with affine model given by $F(x,y)=0$ and projective model given by the homogenous equation
	$$ F(X/Z,Y/Z)Z^d =0.$$
	We can cover this curve with two affine charts $U = \lbrace Z\neq 0 \rbrace$ and $V = \lbrace Y\neq 0\rbrace$ with affine coordinates $x=X/Z, y=Y/Z$ and $s=X/Y,t=Z/Y$ respectively.
	
	We will now compute $H^1(C,\OO)$, $H^0(C,\Omega)$, and a representative of the Kodaira-Spencer class in \v{C}ech cohomology and use this information to write down $\omega \mapsto \KS_{C/K}(\delta) \smile \omega$ in terms of a basis.
	This will be done in three Lemmas.
	
	The presentation we give here is very important because it is well-adapted to this computation. 
	Trying to write down the matrix of the cup-product-with-KS map will fail if we try to write down just some random presentation of $H^0(\Omega)$ and $H^1(\OO)$ because one needs to be able to coerce the image in $H^1(\OO)$ into the particular chosen representatives.
	The presentation in the Lemmas below does this for plane curves.
	\begin{lemma}\label{L:H1(O)}
	In \v{C}ech cohomology with respect to the cover $\lbrace U,V\rbrace$ the vector space $H^0(C,\OO)$ has a basis
			$$\frac{\bar{x}^2}{\bar{y}}, 
		\frac{\bar{x}^3}{\bar{y}},\frac{\bar{x}^3}{\bar{y}^2},
		\frac{\bar{x}^4}{\bar{y}},\frac{\bar{x}^4}{\bar{y}^2},\frac{\bar{x}^4}{\bar{y}^3},
		\ldots,
		\frac{\bar{x}^{d-1}}{\bar{y}},\frac{\bar{x}^{d-1}}{\bar{y}^2},\ldots,\frac{\bar{x}^{d-1}}{\bar{y}^{d-2}}.$$
	\end{lemma}
\begin{proof}
	We can write 
	$$ H^1(C,\OO) = \dfrac{\OO(U\cap V)}{\OO(U) + \OO(V)}.$$
	As $K$-vector spaces we have a direct sum decomposition
	$$\OO(U\cap V) = \bigoplus_{j=0}^{d-1}  K[\bar{y},1/\bar{y}] \bar{x}^j= \bigoplus_{i\in \ZZ}\bigoplus_{j=0}^{d-1}  K \bar{y}^i\bar{x}^j.$$
	Using $\OO(U) = K[\bar{x},\bar{y}]$, and $\OO(V) = K[\bar{s},\bar{t}]=K[\bar{x}/\bar{y},1/\bar{y}]$ decomposing similarly we are left with basis element after cancelling all of the $\bar{x}^i\bar{y}^j$ in the denominator.
\end{proof}.

\begin{lemma}\label{L:H0(Omega)}
	The dual space, $H^0(C,\Omega)$, has a basis given by $\omega_{i,j}$ with $i+j<d-2$ where 
	$$ \omega_{i,j} = x^i y^j \omega_{0,0}, \quad \omega_{0,0} = \dfrac{dx}{\partial F/\partial y} =- \dfrac{dy}{\partial F/\partial x}.$$
\end{lemma}
\begin{proof}
\cite[pg 73]{Hindry2013}
\end{proof}

\begin{lemma}\label{L:prolongations}
Let $\alpha,\beta,\gamma \in K[x,y]$ be such that 	
 $$\alpha \dfrac{\partial F}{\partial x}+ \beta \dfrac{\partial F}{\partial y} + \gamma F = 1$$.
Such a triple $(\alpha,\beta,\gamma)$ exist by the nonsingularity hypothesis of the plane curve $C$.
There is a derivation $D_1:\OO(U)\to \OO(U)$ given by $D_1(x) = -F^{\delta}\alpha$ and $D_1(y)=-F^{\delta}\beta$ prolonging the derivation on $K$.
\end{lemma}
\begin{proof}
	For $D_1 \in T_{C}(U)$ and $D_2\in T_{C}(V)$ prolonging the derivation $\delta$ in $K$ and form $D_{12}=D_1-D_2 \in T_{C/K}(U\cap V)$ which has $[D_{12}] = \KS_{C/K}(\partial) \in H^1(C,T_{X/K})$.
	We know that $D_1(F(x,y))=0$ implies 
	$$F^{\delta}+\dfrac{\partial F}{\partial x}D_1(x) + \dfrac{\partial F}{\partial y}D_1(y)=0,.$$
	By nonsingularity there exists some $\alpha=\alpha(x,y),\beta=\beta(x,y),\gamma=\gamma(x,y)\in K[x,y]$ such that 
	$$ \alpha \dfrac{\partial F}{\partial x}+ \beta \dfrac{\partial F}{\partial y} + \gamma F = 1.$$
	Letting $D_1(x) = -F^{\delta}\alpha$ and $D_2(y)=-F^{\delta}\beta$ we find that 
	$$F^{\delta}+ \dfrac{\partial F}{\partial x} D_1(x) + \dfrac{\partial F}{\partial y}D_1(y) = F^{\delta}(1-\alpha\dfrac{\partial F}{\partial x}-\beta\dfrac{\partial F}{\partial y}) \equiv 0 \mod F(x,y).$$
\end{proof}
	
We will now describe in how to compute the Kodaira-Spencer matrix. 
Computationally (i.e. in \texttt{magma}) the computation will not terminate for a general plane curve (but it will for Picard curves whose special case is described below) due to space restrictions.
The bottleneck is in computing the $\alpha,\beta,\gamma\in K[x,y]$ such that 
$$\alpha \dfrac{\partial F}{\partial x} + \beta \dfrac{\partial F}{\partial y} + \gamma F = 1,$$
which can be done with \texttt{magma}'s \texttt{IdealWithBasis} and \texttt{Coordinates} commands.

The idea is to use the description of $H^1(\OO)$ in Lemma~\ref{L:H1(O)}, the description of $H^0(\Omega)$ in Lemma~\ref{L:H0(Omega)}, together with the description for prolonging derivations in Lemma~\ref{L:prolongations} to write down $D_1$ on the chart $U$ and $D_2$ on the chart $V$ so that on $U\cap V$ the derivation $D_1-D_2$ is the Kodaira-Spencer class which we can then pair using the description of $H^0(\Omega)$.

	For the chart $V$  we define $\widetilde{F}(s,t) = F(s/t,1/t)t^d$ and use that $\widetilde{F}(s,t)=0$ defines the model.
	Using Lemma~\ref{L:prolongations} we there exists some  $\widetilde{\alpha},\widetilde{\beta},\widetilde{\gamma} \in K[s,t]$ be such that 
	$$ \widetilde{\alpha} \dfrac{\partial \widetilde{F}}{\partial s}+ \widetilde{\beta} \dfrac{\partial \widetilde{F}}{\partial t} + \widetilde{\gamma} \widetilde{F} = 1, $$
	and define $D_2(s) = -\widetilde{F}^{\delta}\widetilde{\alpha}$, $D_2(t) = -\widetilde{F}^{\delta}\widetilde{\beta}$ to get $D_2\in T_C(V)$ prolonging the derivation on $K$.
	
	Write $D_{12}$ for $D_1-D_2$ in $T_{C/K}(U\cap V)$. 
	This element represents the Kodaira-Spencer class in \v{C}ech cohomology.
	We compute $D_{12}(x)$ and pair it with each of our forms $\omega_{i,j} =x^iy^j\dfrac{dx}{\partial F/\partial y}$ to get 
	 $$ \KS_{C/K}\smile \omega_{i,j} = \langle x^iy^j\dfrac{dx}{\partial F/\partial y},D_{12}\rangle = \dfrac{x^iy^j}{\partial F/\partial y} D_{12}(x)\in \OO(U\cap V)$$
	which represents the class.
	To finish we need to put $D_{12}(x)$ in terms of $x$ and $y$.
	To do this we use the quotient rule
	$$D_2(x) = D_2(s/t) = \frac{tD_2(s)-sD_2(t)}{t^2} = y\widetilde{F}^{\delta}(-\widetilde{\alpha}(s,t)+s\widetilde{\beta}(s,t)).$$
	To convert we use $s=X/Y = (X/Z)/(Y/Z)=x/y$ and $t=Z/Y=1/y$.
	We have $\widetilde{F}^{\delta}(s,t) = F^{\delta}(s/t,1/t)t^{d-1} = F^{\delta}(x,y)$, $\widetilde{\alpha}(s,t) = \widetilde{\alpha}(x/y,1/y)$, and $\widetilde{\beta}(s,t) = \widetilde{\beta}(x/y,1/y)$.

	\subsubsection{Explicit Kodaira-Spencer Matrix for Picard Curves}\label{S:picard-curves} Earlier in the section, we dealt with a specific family of Picard curves, but next we work more generally. Picard curves are curves of genus $g=3$ where we know that the Jacobian has complex multiplication.
	All such curves have a model of the form
	 $$U \colon y^3 = g(x),$$ 
	 where $g(x)$ is a degree $4$ polynomial over $K$ with non-constant coefficients.
	The complex multiplication is induced by the automorphism $(x,y) \mapsto (x,\zeta_3y)$ where $\zeta_3$ is a primitive third root of unity. 
	For making the chart at infinity later (which we will call $V$) we define $\widetilde{g}(x) = x^4 g(1/x).$ 
	We will suppose that $g(x)$ and $g'(x)$ are coprime, and that $g(x)+\widetilde{g}(x)$ and $g'(x)$ are coprime so there is a smooth projective plane model.
	
	For this curve we look at homogeneous coordinates $[X,Y,Z]$ so that $$C\colon ZY^3=g(X/Z)Z^4$$ 
	is the plane model in $\PP^2_K$. 
	Note that this is degree $d=4$ and $g=3=(d-1)(d-2)/2$.
	The two charts we consider are $U = \lbrace  Z\neq 0\rbrace$ and $V=\lbrace Y\neq 0\rbrace$ as we did for plane curves.
	On the chart $V$ we use the coordinates $s,t$ with $s=X/Y=x/y$, $t=Z/Y=1/y$.
	The model is given by 
	$$V \colon t-\widetilde{G}(s,t)=0$$ where $\widetilde{G}(s,t) = g(s/t)t^4$.
	Away from the $U$ chart there is a single point "at infinity" which is $[0,1,0]$. 
	
	We now compute the ingredients for our map following the previous example for plane curves. 
	By Lemma~\ref{L:H1(O)} we have a basis for $H^1(C,O)$ given by
	$$v_1=[x^2/y], \quad v_2=[x^3/y], \quad v_3=[x^3/y^2].$$

	By Lemma~\ref{L:H0(Omega)} we have a basis for $H^0(\Omega)$ given by 
	 $$ \omega_1 = dy/3y^2 = dx/g'(x), \quad \omega_2 = x \omega_1, \quad \omega_3 = y \omega_1.$$
	 	
	We now proceed as in Lemma~\ref{L:prolongations} in order to define derivations $D_1,D_2$ so that $D_1-D_2 \in T_{C/K}(U\cap V)$ give a representative for the Kodaira-Spencer class.
	As in Lemma~\ref{L:prolongations} we need a number of auxillary quantities: $\alpha_1,\beta_1,\alpha_3,\beta_3, \widetilde{G}, \widetilde{G}^{\delta}, \partial \widetilde{G}/\partial t, \partial \widetilde{G}/\partial s$ and we need them all written in $(x,y)$ coordinates. 
	These are collected in the next paragraph.

	By our coprimailty assumptions there exists some $\alpha_1(x),\beta_1(x)\in K[x]$ such that $$\alpha_1(x)g(x) + \beta_1(x)g'(x)=1.$$
	Similarly there exists some $\alpha_3(x), \beta_3(x)\in K[x]$ such that $\alpha_3(x)( g(x) + \widetilde{g}'(x) ) +\beta_3(x) g'(x)=1$.
	Let $g^{\delta}(x)$ be the polynomial obtained by taking the derivatives of the coefficients. 
	It is useful to make the following computations:
	$$\widetilde{G}= g(x)/y^4, \quad \widetilde{G}^{\delta} = g^{\delta}(s/t)t^4 = g^{\delta}(x)/y^4$$
	$$\partial \widetilde{G}/\partial t = g'(s/t) t^3 = g'(x)/y^3, \quad \partial \widetilde{G}/\partial s =\widetilde{g}'(s/t) t^3 = \widetilde{g}'(x)/y^3$$
	
	With everything in hand we make a computation analogous to Lemma~\ref{L:prolongations} for general plane curves.
	\begin{lemma}
		\begin{enumerate}
			\item On the chart $U$, there is a derivation $D_1$ prolonging the derivation on $K$ defined by 
			$$D_1(x) = -g^{\delta}(x)\beta_1(x), \quad D_1(y) = g^{\delta}(x)\alpha_1(x)$$. 
			\item On the chart $V$, there is a derivation $D_2$ prolonging the derivation on $K$ defined by 
			$$D_2(t) = \widetilde{G}^{\delta}y^3\alpha_3(x), \quad D_2(s) = -\widetilde{G}^{\delta}y^3\beta_3(x).$$
		\end{enumerate}
	\end{lemma}
	
	This then leads to our computation of the Kodaira-Spencer map. In order to determine the linear combination 
	$$\eta=a \omega_1 + b \omega_2 + c \omega_3 $$
	for $a,b,c$ in $K$ defining a global differential form in $H^0(\Omega)$ that lifts to a global KS-form it needs to have a cup product with Kodaira-Spencer which is zero. 
	
	On $U \cap V$ we have $D_1 - D_2 := D_{12}$ defining the Kodaira-Spencer class. 
	When we take $\eta(D_{12})$ it is going to give us an element of $\OO(U \cap V)$ and then we write this in terms of our basis of $H^1(O)$ we will get the image of a particular form under cup product with Kodaira-Spencer. 
	I guess to get the matrix one needs to compute it might be best to compute this for $\omega_1, \omega_2,$ and $\omega_3$ separately. 
	We are now ready to compute,
	\begin{align*}
	\omega_1(\KS_{X/K}(\delta) ) &=  \langle\dfrac{dx}{y^2}, D_{12}\rangle \\ 
	&= \frac{1}{y^2}D_{12}(x)\\
	&= \frac{1}{y^2}( -g^{\delta} \beta_1 +g^{\delta} (\alpha_3+x \beta_3) ) \\
	&= -\frac{g^{\delta} }{y^2}(\beta_1 - \alpha_3-x \beta_3).
	\end{align*}
	Here to compute $D_2(x)$ we used the quotient rule
	$$D_2(x) = D_2(s/t) = ( D_2(s) t - D_2(t) s)/t^2 = -g^{\delta}(x)( \alpha_3 + x \beta_3).$$ 
	We can now state how the computation in Example~\ref{E:picard}.
	\begin{theorem}
		\begin{enumerate}
			\item
	Let $M=-g^{\delta} (\beta_1 - \alpha_3-x \beta_3).$
	Let $A$ be the coefficient of $x^3$ in $M$.
	Let $B$ be the coefficient of $x^2$ in $M$.
	Then
	$$\KS_{C/K}(\delta) \smile \omega_1= Av_3,$$
	$$\KS_{C/K}(\delta)\smile  \omega_2 = Bv_3,$$
	$$\KS_{C/K}(\delta)\smile \omega_3 = Bv_1+Av_2,$$ 
	where $v_1=[x^2/y]$, $v_2=[x^3/y]$, and $v_3=[x^3/y^2]$ are the basis of $H^1(\OO)$.
	\item 	With $A$ and $B$ as above then $K$-linear multiples of the form
	$$\eta = B \omega_1 - A \omega_2 =(B - A x)dx/3y^2,$$
	lift to global KS-forms on the Picard curve $C$.
	\end{enumerate}
	\end{theorem}

\begin{remark}
	For a general Picard curve $y^3 = g(x)$ where
	 $$ g(x) = x^4+ax^3+bx^2+cx+d $$
	for $a,b,c,d$ differential variables one can explicitly compute $A$ and $B$ in \texttt{magma} but the result is a horrifying expression in of $a,b,c,d,a',b',c',d'$ which doesn't fit on the page in any reasonable way.
\end{remark}

	\subsection{$\delta$-Stratification of the Moduli Space of Principally Polarized Abelian Varieties}
	The Kodaira-Spencer rank and the existence of global new forms is closely related to Buium's work on the moduli space of Abelian varieties.  
	\begin{theorem}[{\cite[Thm 6.2 (6)]{Buium1993}}]
		Let $K$ be a differentially closed field. 
		Let $n\geq 3$ and $g\geq 2$. 
		Let $\mathcal{A}_{g,n}$ be the moduli space of principally polarized abelian varieties with level $n$ structure. 
		For each $r$ the collection $\Wcal_{r,g,n}\subset \mathcal{A}_{g,n}(K)$ consisting of principally polarized abelian varieties of Kodaira-Spencer rank less than or equal to $r$ forms a $\delta$-subvariety cut out by order one equations in $\mathcal{A}_{g,n}$ and these fit into a stratification
		$$ \Wcal_{0,g,n} \subset \Wcal_{1,g,n} \subset \cdots \subset \Wcal_{g,g,n} = \mathcal{A}_{g,n}(K),$$
		all of which are non-empty. In particular, there exists a $\delta$-open on which all abelian varieties have maximal Kodaira-Spencer rank.
	\end{theorem}

	It is perhaps useful to observe that in virtue of Theorem~\ref{L:descent} that $\Wcal_{0,g,n}$ consist of those points in the moduli space whose corresponding abelian varities admit models over the constants.

	Now to show the existence of new KS-forms (and exhibit new nonorthogonality classes of strongly minimal sets in $DCF_0$), it suffices, by Theorem~\ref{T:manin-kernels-and-tau-forms}, to exhibit a nonisotrivial curve with simple Jacobian and  intermediate Kodaira-Spencer rank. This is equivalent to showing that the Torelli locus intesects $\Wcal_{g-1,g,n}\setminus \Wcal_{0,g,n}$.
	
	This doesn't follow from Buium's techniques; he proved that $\Wcal_{j,g,n}\setminus \Wcal_{j-1,g,n}\neq \emptyset$ for $1\leq j\leq g$ using products of elliptic curves. 

 The following questions are extremely natural both from the perspective of moduli problems and model theoretic questions about $DCF_0$. 
	\begin{enumerate}
		\item For each $j$ with $1\leq j \leq g$ is it the case that $\Wcal_{j,g,n}\setminus \Wcal_{j-1,g,n}$ contains a simple abelian variety?
		\item For each $j$ with $1\leq j \leq g$ is it the case that $\Wcal_{j,g,n}\setminus \Wcal_{j-1,g,n}$ intersects the Torelli locus?
	\end{enumerate}
	
	\subsection{Decomposition of subspaces with CM and Kodaira-Spencer ranks}\label{S:CM-and-KS}
	Fix a CM type $\Phi = \lbrace \phi_1,\ldots,\phi_e \rbrace$ for $F$ so that 
	$$\Phi \cup \overline{\Phi} = \lbrace \phi_1, \overline{\phi}_1, \phi_2, \overline{\phi}_2, \ldots, \phi_e,\overline{\phi}_e \rbrace,$$ 
	are a complete set of embeddings $F \hookrightarrow \CC$. 
	One has 
	$$F \otimes_{\QQ} \CC \cong \bigoplus_{\phi \in \Phi \cup \overline{\Phi}} \CC(\phi)$$ as $\CC$-vector spaces where the $F$-algebra structure on $\CC(\phi) = \CC 1_{\phi}$ is given by $a\cdot 1_{\phi} = \phi(a)1_{\phi}$ for $a\in F$ where $1_{\phi}$ is the unit of the ring. 
	Also $\CC(\phi) = F\otimes_{F,\phi} \CC$
	
	Every complex vector space $W$ which has an $F$-algebra structure is automatically a $\CC\otimes_{\QQ} F$-module.
	We have as $\CC\otimes F$-modules 
	$$W \cong \bigoplus_{\phi \in \Phi \cup \overline{\Phi}} W_{\phi},$$ 
	and $W_{\phi}$ is a subspace where $a \in F$ acts via the embedding $\phi$. 
	
	We now review the decomposition from \cite[\S4]{Andre2017} of the cup-product-with-Kodaira-Spencer map given extra endomorphisms.
	Let $B$ be a smooth variety over $\CC$.
	Let $\mathcal{A}/B$ be a nonisotrivial $g$-dimensional abelian scheme with rational endomorphism algebra isomorphic to $F$, a CM field. 
	Let $\Phi$ be a CM type for $F$. 
	As in the previous section the complex of $\OO_B$-modules $\Omega_{\mathcal{A}/B}^{\bullet}$ defining sheafy algebraic de Rham cohomology $\underline{H}^1_{\dR}(\mathcal{A}/B)$ decomposes as
	$$ \Omega_{\mathcal{A}/B}^{\bullet} = \bigoplus_{\phi \in \Phi \cup \overline{\Phi}} \Omega_{\mathcal{A}/B,\phi}^{\bullet},$$
	which induces a decomposition of the variation of Hodge structure so that 
	$$ \theta_{\partial}: \Omega_{\mathcal{A}/B}^1 \to \underline{H}^1_{\dR}(\mathcal{A}/B)/ \Omega_{\mathcal{A}/B}^1, \quad \theta_{\partial}(\omega) = \omega \cup \KS_{\mathcal{A}/B}(\partial), $$
	decomposes $\theta_{\partial} = \oplus_{\phi \in \Phi \cup \overline{\Phi}} \theta_{\partial,\phi}$.
	\begin{definition}
		For $\phi \in \Phi \cup \overline{\Phi}$ and $b\in B(\CC)$ 
		$$ r(\phi) = \dim_{\CC}(\Omega_{\mathcal{A}/B, \phi, b}^1), \quad s(\phi) = \dim_{\CC} (\underline{H}^1_{\dR}(\mathcal{A} / B)_{\phi,b}/ \Omega_{\mathcal{A}/B, \phi, b}^1).$$
		The collection $(r(\phi),s(\phi))$ for $\phi \in \Phi$ is called the \emph{Shimura type} of the family $\mathcal{A}/B$.
	\end{definition}
	It is a fact that the Shimura type is independent of the base points $b\in B(\CC)$.
	Also, by duality we have $r(\phi) = s(\overline{\phi})$ for each $\phi \in \Phi \cup \overline{\Phi}$ (see \cite[\S4]{Andre2017}). 
	The following then follows which we record for later purposes,
	\begin{lemma}\label{L:KS-rank-bound}
		We now have 
		\begin{equation}
		\rk(\theta_{\partial}) \leq 2 \sum_{\phi \in \Phi}\min\lbrace r(\phi),s(\phi)\rbrace
		\end{equation}
	\end{lemma}
	\begin{proof}
		This is because the rank of the linear maps $\theta_{\partial,\phi}$ for $\phi \in \Phi \cup \overline{\Phi}$ is bounded by the dimension of the domain and the dimension of the codomain and that the rank of $\theta_{\partial}$ is equal to the sum of the ranks of its summands.
		We have for $\phi \in \Phi$ we have $\rk(\theta_{\partial,\phi}) \leq \max\lbrace r(\phi),s(\phi)\rbrace$, we also have $\rk(\theta_{\partial,\overline{\phi}}) \leq \max\lbrace r(\overline{\phi}),s(\overline{\phi})\rbrace = \max \lbrace s(\phi),r(\phi) \rbrace.$ 
		The result follows.
	\end{proof}
	
	\begin{remark}
		We can also describe these numbers in terms of modules on the generic fiber of thi family. 
		Let $K$ be the algebraic closure of $\kappa(B)$. 
		Let $A = \mathcal{A}_K$ be the geometric generic fiber of this family.
		$$r(\phi) = \dim_K( H^0(A,\Omega_{A/K}^1)_{\phi}), \quad s(\phi) = \dim_K( H^1(A,\OO)_{\phi}).$$
	\end{remark}
	
	On can see now that the description of Shimura type from Andr\'{e} given above coincides with the description of \cite[\S9.6]{Birkenhake2004}.
	For them, the $r(\phi)$ and $s(\phi)$ are defined representation theoretically. 
	Let $\Phi = \lbrace \phi_1,\ldots,\phi_e \rbrace$ be a CM type. 
	It is a fact that every complex representation $\rho: F\to M_g(\CC)$ is conjugate to one of the form 
	$$ 
	a\mapsto \begin{pmatrix}
	\phi_1(a) I_{r(\phi_1)} & & & & \\
	& \overline{\phi_1}(a) I_{r(\overline{\phi}_1)} & & & \\
	& & \ddots & & \\
	& & & \phi_e(a) I_{r(\phi_e)}& \\
	& & & & \overline{\phi}_e(a) I_{r(\overline{\phi}_e)}\\
	\end{pmatrix}
	$$ 
	where the $I_k$ are $k\times k$.
	Since $r(\overline{\phi}) = s(\phi)$ we see that $\sum_{\phi\in \Phi}( r(\phi) + s(\phi))=g$.
	We also have $\sum_{\phi \in \Phi \cup \overline{\Phi}}  r(\phi)+s(\phi) = 2g$, $r(\phi)+s(\phi)=m$ where $g=em$. 
	In the parlance of Hodge there the $r(\phi)$ and $s(\phi)$ are invariants of the action of $F$ on the pure Hodge structure  $H^0(\mathcal{A}_b,\Omega_{A/B}^1)$ of type $(0,1)$ in the variation of Hodge structures.

	\subsection{Applications to CM Abelian Varieties}
	\begin{theorem}\label{T:intermediate-KS-ranks}
		Let $A/K$ be a nonisotrivial simple abelian variety of dimension $g$. 
		Suppose that $\End(A/K)_{\QQ}\cong F$ where $F$ is a CM field with $[F:F^+]=e$.
		If $g/e$ is odd then 
		$$1\leq \rk(\KS_{A/K}(\partial))\leq g-e.$$
In particular if $g$ is not divisible by $2$ then as long as $\End(A/K)_{\QQ}$ is a CM-field, it follows that $A$ has a Manin kernel of intermediate rank.
\end{theorem}
	\begin{proof}
		This is a direct application of Lemma~\ref{L:KS-rank-bound}.
		If $m=g/e$ is odd and $r(\phi)+s(\phi)=m$ then $\min\lbrace r(\phi),s(\phi)\rbrace \leq (m-1)/2.$
		This then proves that 
		$$ \rk(\KS) \leq 2\sum_{\phi \in \Phi} \max\lbrace r(\phi),s(\phi)\rbrace \leq 2\sum_{\phi \in \Phi} (m-1)/2 \leq e(m-1) = g-e. $$
	\end{proof}
	
	Fix a positive integer $g$.
	Let $F$ be CM field and let $[F^+:F]=e$. 
	Suppose that $g=em$ for some integer $m$.
	For every $(r_j,s_j)$ with $r_j+s_j=m$ admits moduli \cite[9.6.4]{Birkenhake2004} so that the representation is decribed as in section~\ref{S:CM-and-KS}.
	In fact abelian varieties are parameterized by arithmetic quotients of the symmetric spaces $$\Hcal_{r_1,s_1} \times \Hcal_{r_2,s_2} \times \cdots \times \Hcal_{r_e,s_e} $$
	and these locally symmetric spaces have dimension $\sum_{j=1}^e r_js_j$. 
	This means as long as $r_js_j\neq 0$ for some $j$ there exist positive dimensional algebraic famillies of simple abelian varieties over $\CC$ with CM by $F$ and Shimura type $(r_j,s_j)$. 
	This also means we are free to select our $r_j$ and $s_j$ for $j=1,\ldots, e$ to constrain our moduli.
	
	\begin{theorem}
		For each $g$ and $F$ as above there a differential field $(K,\delta)$ and simple abelian variety $A/K$ with $1\leq \rk(\KS_{A/K}(\delta))\leq 2$.
	\end{theorem}
	\begin{proof}
		Choose an abelian variety with $r_1=1,s_1=m-1$ then let $r_js_j=0$ for $j=2,\ldots, e-1$.
		Then Lemma~\ref{L:KS-rank-bound} gives the result.
	\end{proof}
	This is as tight as a bound we can create with the methods of this paper. 
	
	\subsection{Applications to Curves}\label{S:applications-to-curves}
	The first examples of families of CM Jacobians, are $g=3$ Picard curves.
	These are curves of the form 
	$$ C\colon \quad y^3=f(x), \qquad f\in k[x], \ \deg(f)=4. $$
	where $k$ is a field. 
	In this case there is a CM field $F$ action where $\QQ(\zeta_3) \subset F$ and $F=F^+(\zeta_3)$ for $F^+$ a totally real subfield and $(x,y)\mapsto (x,\zeta_3y)$ gives the automorphism. 
	We can create families of by varying $f$ and then by taking the geometric generic fiber to get some $C/(K,\delta)$ with $\Jac_C$ a CM Jacobian.
	Since $e \vert g$ we have $e=3$ or $e=1$. 
	In the case that $e=3$ we have $m=1$ and they are rigid since for each $j$ we have $r_j+s_j=1$ which forces $\min\lbrace r_j,s_j\rbrace=0$ and hence $\dim(\Hcal_{r_j,s_j}) = r_js_j=0$. 
	In the case that $e=1$ we have $m=3$ and Theorem~\ref{T:intermediate-KS-ranks} applies. 
	These are $\QQ(\zeta_3)$ families.

	Conversely, every principally polarized abelian variety over $\overline{k}$ with $\End(A/\overline{k})= \OO_F$ and $\QQ(\zeta_3) \subset F$ then $A$ is simple and $A = \Jac_C$ for some Picard curve $C$ \cite[Lemma 1]{Koike2005} (the surrounding discussion of loc. cit. is also informative).
	Hence the $\QQ(\zeta_3)$-family are necessarily simple Jacobians of a family of Picard curves. 
	
	\begin{theorem}
		There geometric generic fiber of the family of Picard curves gives some $C/K$ with $1\leq\rk(\KS_{C/K}(\delta))\leq 2$.
		Hence on $C/K$ there exists some new form $\eta \in H^0(C,\Omega^{\tau})$ and some $\Sigma = (C,\eta)^{\sharp} \subset C(\widehat{K})$ an order one differential equation which is strictly disintegrated.
	\end{theorem}
	\begin{proof}
		In this family and $g=3$ so Theorem~\ref{T:intermediate-KS-ranks} gives the result.
	\end{proof}

	The main result of \cite{Frediani2015} (extending work of \cite{Moonen2010} and \cite{Jong1991}) is that there exist nonisotrivial families $C\to B$ over the complex numbers of genera $\leq 7$ with complex multiplication Jacobians.
	Because we know the rational endomorphism algebras of the geometric generic fibers are simple, we get to conclude that these examples are simple.
	\begin{theorem}
		There exist nontrivial isotrivial curves $C/(K,\partial)$ of genus $g=5,7$ with intermediate Kodaira-Spencer rank.
		As a consequence there exists new forms $\eta$ on the corresponding curves and strictly disintegrated order one differential algebraic varieties $\Sigma = (C,\eta)^{\sharp} \subset C(\widehat{K})$. 
	\end{theorem}
	\begin{proof}
		In these examples the genus $g$ is prime so  Theorem~\ref{T:intermediate-KS-ranks} gives the result.
	\end{proof}

	\bibliographystyle{amsalpha}
	\bibliography{orderone}

	\appendix
	
	\section{Order one ODEs with constant coefficients and differential forms} \label{HrItapp}
	
	Let $(K,\delta)$ be a differential field.
	Let $C$ be a smooth projective curve defined over the constants $K^{\delta}$.
	The derivation $\delta$ gives a section $\exp_1:C \to T_C$, and so one can take a differential form $\eta \in \Omega^1_{K(C)/K}$ and study
	\begin{equation}\label{E:HI-sharps}
	(C, \eta)^ \sharp(\Khat) = \{ x \in C (\widehat{K}) \, | \, \eta(\exp_1(x)) = 1 \}.
	\end{equation}
	
	It is the case that for every irreducible order one differential variety, except the variety of constants $\lbrace x \in\Khat \colon \delta x =0 \rbrace$, can be written in the form $(C, \eta )^ \sharp $ up to finitely many points (see \cite[pg 4274]{Hrushovski2003})
	
	\begin{proposition*}[{\cite[Prop 2.1]{Hrushovski2003}}]
		Let $C$ be a curve over a differential field $K$. 
		If $\omega \in \Omega_{K(C)/K}^1$ new, then $(C,\omega)^\sharp$ is geometrically trivial.
	\end{proposition*}
	
	\begin{lemma*}[{\cite[Lem 2.13]{Hrushovski2003}}]\label{esse}  Let $C /K^{\delta}$ be a curve of genus at least two. 
		There is a countable union $S = \cup _{l\geq 0} S_l$ of proper subspaces of $H^0 ( \Omega^1 , C)$ such that any $1$-form outside of $S$ is essential. 
	\end{lemma*} 
	
	This result gives examples of many trivial strongly minimal systems living on curves over $K^{\delta}$ of genus greater than or equal to $2$. 
	
\end{document}